\pdfoutput=1
\RequirePackage{amsmath}
\RequirePackage{amssymb}
\RequirePackage{amsthm}
\documentclass[pdflatex,sn-mathphys-num]{sn-jnl-modified}
\usepackage{graphicx}
\usepackage{mathrsfs}
\usepackage{xcolor}
\usepackage{mathtools}
\usepackage{cases}
\usepackage{empheq}
\usepackage{lmodern}
\usepackage{physics}
\usepackage{braket}
\usepackage[short]{optidef}
\usepackage[shortlabels]{enumitem}
\usepackage{draftwatermark}
\SetWatermarkText{\textup{\textup{\textup{PREPRINT}}}}
\SetWatermarkScale{0.6}
\SetWatermarkColor[gray]{0.85}
\usepackage{orcidlink}
\newlist{mylist}{enumerate}{1}
\setlist[mylist,1]{label = \textup{(\alph*)},left = 0pt}   
\DeclareMathOperator{\relint}{int}
\DeclareMathOperator{\relbdy}{\partial\!}
\DeclareMathOperator{\linspan}{span}
\DeclareMathOperator{\conv}{conv}
\DeclareMathOperator{\diag}{diag}
\DeclareMathOperator{\cliquenum}{\omega}
\DeclareMathOperator{\supp}{supp}
\DeclareMathOperator{\bary}{bary}

\DeclareMathOperator{\Lagr}{\mathcal{L}}
\DeclareMathOperator{\operatorgkkt}{g\mathcal{KKT}}
\DeclareMathOperator{\operatorkkt}{\mathcal{KKT}}
\newcommand{\kkt}[1]{\operatorkkt(#1)}
\newcommand{\gkkt}[1]{\operatorgkkt(#1)}
\newcommand{\kktsubgraph}[2]{\operatorkkt\left(#1, #2\right)}
\newcommand{\gkktsubgraph}[2]{\operatorgkkt\left(#1, #2\right)}
\newcommand{\tuplenum}{\mathbb{E}}
\newcommand{\bbN}{\mathbb{N}}
\newcommand{\bbZ}{\mathbb{Z}}

\newcommand{\bbR}{\mathbb{R}}
\newcommand{\QuadForm}[2]{{#2}\! ^\top \! {#1} {#2}}
\let\vec\relax
\newcommand{\vec}[1]{\vb*{#1}}
\newcommand{\mtx}[1]{\vb*{#1}}
\newcommand{\intupto}[1]{[#1]}
\newcommand*{\card}[1]{\lvert #1 \rvert}
\theoremstyle{thmstyleone}
\newtheorem{theorem}{Theorem}
\newtheorem{corollary}{Corollary}
\newtheorem{lemma}{Lemma}
\newtheorem{proposition}{Proposition}
\newtheorem{definition}{Definition}
\ifx \bisbn   \undefined \def \bisbn  #1{ISBN #1}\fi
\ifx \binits  \undefined \def \binits#1{#1}\fi
\ifx \bauthor  \undefined \def \bauthor#1{#1}\fi
\ifx \batitle  \undefined \def \batitle#1{#1}\fi
\ifx \bjtitle  \undefined \def \bjtitle#1{#1}\fi
\ifx \bvolume  \undefined \def \bvolume#1{\textbf{#1}}\fi
\ifx \byear  \undefined \def \byear#1{#1}\fi
\ifx \bissue  \undefined \def \bissue#1{#1}\fi
\ifx \bfpage  \undefined \def \bfpage#1{#1}\fi
\ifx \blpage  \undefined \def \blpage #1{#1}\fi
\ifx \burl  \undefined \def \burl#1{\textsf{#1}}\fi
\ifx \doiurl  \undefined \def \doiurl#1{\url{https://doi.org/#1}}\fi
\ifx \betal  \undefined \def \betal{\textit{et al.}}\fi
\ifx \binstitute  \undefined \def \binstitute#1{#1}\fi
\ifx \binstitutionaled  \undefined \def \binstitutionaled#1{#1}\fi
\ifx \bctitle  \undefined \def \bctitle#1{#1}\fi
\ifx \beditor  \undefined \def \beditor#1{#1}\fi
\ifx \bpublisher  \undefined \def \bpublisher#1{#1}\fi
\ifx \bbtitle  \undefined \def \bbtitle#1{#1}\fi
\ifx \bedition  \undefined \def \bedition#1{#1}\fi
\ifx \bseriesno  \undefined \def \bseriesno#1{#1}\fi
\ifx \blocation  \undefined \def \blocation#1{#1}\fi
\ifx \bsertitle  \undefined \def \bsertitle#1{#1}\fi
\ifx \bsnm \undefined \def \bsnm#1{#1}\fi
\ifx \bsuffix \undefined \def \bsuffix#1{#1}\fi
\ifx \bparticle \undefined \def \bparticle#1{#1}\fi
\ifx \barticle \undefined \def \barticle#1{#1}\fi
\bibcommenthead
\ifx \bconfdate \undefined \def \bconfdate #1{#1}\fi
\ifx \botherref \undefined \def \botherref #1{#1}\fi
\ifx \url \undefined \def \url#1{\textsf{#1}}\fi
\ifx \bchapter \undefined \def \bchapter#1{#1}\fi
\ifx \bbook \undefined \def \bbook#1{#1}\fi
\ifx \bcomment \undefined \def \bcomment#1{#1}\fi
\ifx \oauthor \undefined \def \oauthor#1{#1}\fi
\ifx \citeauthoryear \undefined \def \citeauthoryear#1{#1}\fi
\ifx \endbibitem  \undefined \def \endbibitem {}\fi
\ifx \bconflocation  \undefined \def \bconflocation#1{#1}\fi
\ifx \arxivurl  \undefined \def \arxivurl#1{\textsf{#1}}\fi
\csname PreBibitemsHook\endcsname

\begin{document}

\title{On Generalized KKT Points for the Motzkin-Straus Program}

\author*[1,2]{\fnm{Guglielmo} \sur{Beretta \orcidlink{0000-0002-4757-1965}
} }\email{guglielmo.beretta@unive.it}
\equalcont{These authors contributed equally to this work.}

\author[1,3,4]{\fnm{Alessandro} \sur{Torcinovich \orcidlink{0000-0001-8110-1791}
}}\email{alessandro.torcinovich@unibz.it}
\equalcont{These authors contributed equally to this work.}

\author[1,5,6]{\fnm{Marcello} \sur{Pelillo \orcidlink{0000-0001-8992-9243}
}}\email{pelillo@unive.it}

\affil*[1]{\orgdiv{DAIS}, \orgname{Ca' Foscari University of Venice}, \orgaddress{\street{Via Torino 155}, \postcode{30170} \city{Venezia}, \country{Italy}}}

\affil[2]{\orgdiv{DAUIN}, \orgname{Polytechnic University of Turin}, \orgaddress{\street{Corso Castelfidardo 34/d}, \postcode{10138} \city{Torino}, \country{Italy}}}

\affil[3]{\orgdiv{D-INFK}, \orgname{ETH}, \orgaddress{\street{Universit\"{a}tstrasse 6}, \city{8006 Z\"{u}rich}, \country{Switzerland}}}

\affil[4]{\orgdiv{Faculty of 
Engineering}, \orgname{Free University of Bolzano-Bozen}, \orgaddress{\street{Via 
Buozzi 1}, \postcode{39100} \city{Bolzano}, \country{Italy}}}

\affil[5]{\orgdiv{College of Mathematical Medicine}, \orgname{Zhejiang Normal University}, \orgaddress{\street{688 Yingbin Road}, \city{Jinhua}, \state{Zhejiang Province},
 \postcode{321004} \country{China}}}

\affil[6]{\orgname{European Centre for Living Technology}, \orgaddress{\street{Ca' Bottacin, Dorsoduro 3911, Calle Crosera}, \postcode{30123} \city{Venezia}, \country{Italy}}}
\abstract{
    In 1965, T.~S.~Motzkin and E.~G.~Straus established an elegant connection between the clique number of a graph and the global maxima of a quadratic program defined on the standard simplex. Over the years, this seminal finding has inspired a number of studies aimed at characterizing the properties of the (local and global) solutions of the Motzkin-Straus program. The result has also been generalized in various ways and has served as the basis for establishing new bounds on the clique number and developing powerful clique-finding heuristics. Despite the extensive work done on the subject, apart from a few exceptions, the existing literature pays little or no attention to the Karush-Kuhn-Tucker (KKT) points of the program. In the conviction that these points might reveal interesting structural properties of the graph underlying the program, this paper tries to fill in the gap. In particular, we study the generalized KKT points of a parameterized version of the Motzkin-Straus program, which are defined via a relaxation of the usual first-order optimality conditions, and we present a number of results that shed light on the symmetries and regularities of certain substructures associated with the underlying graph.
    These combinatorial structures are further analyzed using barycentric coordinates, thereby providing a link to a related quadratic program that encodes local structural properties of the graph. This turns out to be particularly useful in the study of the generalized KKT points associated with a certain class of graphs that generalize the notion of a star graph. Finally, we discuss the associations between the generalized KKT points of the Motzkin-Straus program and the so-called replicator dynamics, thereby offering an alternative, dynamical-system perspective on the results presented in the paper.
}    



\maketitle
\section{Introduction}\label{sec:intro}
In 1965, T.~S.~Motzkin and E.~G.~Straus published a study \cite{motzkin_straus_1965} on the quadratic program (QP)
\begin{maxi}|l|
{\vec{x} \in \Delta_n}
{f(\vec{x}) =\QuadForm{\mtx{A}}{\vec{x}},}
{ \label{prog:ms}}
{ }
\end{maxi}
where $\mtx{A}$ is the adjacency matrix of a graph $G$ with $n$ nodes, and $\Delta_n$ denotes the \emph{standard simplex} in $\bbR^{n}$, i.e., the convex hull of the canonical basis of $\bbR^n$. Motzkin and Straus discovered that the value of Program~\eqref{prog:ms} is $1 - {\cliquenum(G)}^{-1}$, where $\cliquenum(G)$ stands for the clique number of $G$, and they also provided an explicit formula to map every maximum clique of $G$ to a global solution for Program~\eqref{prog:ms}, so that distinct cliques correspond to distinct solutions. This correspondence naturally extends to an injection that maps every non-empty set of nodes of $G$ to an element of $\Delta_n$ and that can be used \cite{pelillo_jagota_1996,tang_maxima_2022} to characterize (strict) local solutions of Program~\eqref{prog:ms} in terms of (strictly) maximal cliques of $G$.

However, as discussed e.g. in \cite{pelillo_jagota_1996}, the Motzkin-Straus QP can exhibit \emph{spurious} solutions, i.e., optima that do not correspond to any maximum clique. This may have a negative impact on applications, such as the one pioneered in \cite{pardalos_global_1990}, where Motzkin and Straus's discoveries are used to find maximum cliques. To address this issue, scholars investigated modified versions of the Motzkin-Straus QP, see e.g. \cite{bomze_1997,gibbons_continuous_1997,hungerford_rinaldi_2017}. In particular, in \cite{bomze_1997} the Motzkin-Straus QP was modified by incorporating a regularization term into its objective function, and this adjustment resulted in a program with local (global) solutions that are in one-to-one correspondence with maximal (maximum) cliques.

Since its introduction, the Motzkin-Straus QP and its variations have been extensively explored in the literature, inspiring various bounds on the clique number, see e.g. \cite{wilf_spectral_1986,bbpp_maximum_1999,nikiforovInequalitiesLargestEigenvalue2002,budinich_exact_2003,stozhkov_new_2017}, and a number of heuristics for the maximum clique problem \cite{pardalos_global_1990,gibbons_continuous_1997,bomze_approximating_2000,bomze_annealed_2002,pelillo_torsello_payoff_monotonic_2006,daniluk_ImplementationMaximumClique_2019,bonamiSolvingQuadraticProgramming2019} as well as for extensions of the maximum clique problem in the presence of weights on the nodes \cite{gibbons_continuous_1997,bbpp_maximum_1999,bomze_approximating_2000}, or 
on the edges \cite{pavan_dominant_2007,rota_bulo_dominant-set_2017},
or in the case of hypergraphs \cite{rota_bulo_generalization_2009}.

As a natural consequence of Motzkin and Straus's findings, most of these studies focused on Motzkin-Straus QP primarily due to the properties of its optima.
However, little to no attention was put on a broader set of points, namely the Karush-Kuhn-Tucker (KKT) points \cite{luenberger_linear_2016} of the program, and their connection with specific combinatorial structures that are associated with these points --- with some exceptions, see e.g. \cite{gibbons_continuous_1997,bomzeRegularityDegeneracyDynamics2002,tang_maxima_2022}.

This paper is the outcome of an exploratory study aimed at investigating additional combinatorial properties related to the Motzkin-Straus QP. To pursue this task, we considered a parametric version of the program, which is due to Bomze \cite{bomze_1997} and was further analyzed in \cite{bomze_annealed_2002,pelillo_torsello_payoff_monotonic_2006},
and we performed a theoretical analysis pertaining to its generalized KKT points \cite{bomze_1997,bomzeRegularityDegeneracyDynamics2002}, defined by a relaxation of the usual KKT conditions.

In this paper, we present new results linking the properties of these points with those of certain substructures of the graph underlying the program, including symmetries and regularity.
After introducing the reader to the parametrized program and its generalized KKT points, we extend known results that link these points and regular induced subgraphs, and we describe how a generalized KKT point is related to the symmetries, i.e., automorphisms, of the subgraph induced by its support. 
To further analyze the combinatorial structures associated with these points, we leverage the notion of barycentric coordinates \cite{hormann_2017_GBC,quarteroni_2017_NMfDP} to obtain a convenient representation of points in the standard simplex, and study some conditions such that these new coordinates, when related to generalized KKT points, satisfy the KKT conditions for another QP encoding the local structure of the graph.
Under additional hypotheses that involve equitable partitions \cite{godsil_agt_2001}, it is possible to obtain stronger results, valid in particular for a class of graphs that generalizes star graphs \cite{diestel_2005_graph_theory}.
Finally, we discuss a connection with replicator dynamics, a model developed in evolutionary game theory \cite{taylor_1978_ESS,Hofbauer_Sigmund_Egap_1998,weibull_1995_EGT} that found applications in diverse fields and is related to standard quadratic optimization \cite{bomze_1997,bomzeRegularityDegeneracyDynamics2002}.
Thanks to this connection, the results presented in the paper admit an alternative, dynamical-system-related interpretation.
A shorter version of this work can be found in \cite{beretta_generalized_2024}.
\section{Notation}\label{sec:notation}
For $n \in \bbN$, we set $\intupto{n} = \set{ i \in \bbN : 1 \leq i \leq n}$.
We denote by $\vec{0}$ (resp. $\vec{1}$) a vector with every component equal to $0$ (resp. $1$), with dimension in agreement with the context in which it is used.
In this paper, $\diag(r_1, r_2, \dots, r_n)$ denotes the $n \times n$ diagonal matrix having on its main diagonal $r_i$ as its $i$-th element.
Given $\vec{x}_1$, $\vec{x}_2$, $\dots$, $\vec{x}_k \in \bbR^{n}$, their \emph{convex hull} $\conv(\vec{x}_1, \vec{x}_2, \dots, \vec{x}_k)$ is the set $\set{ \sum_{m = 1}^k y_m \vec{x}_m : \vec{y} \in \Delta_k}$, which is the smallest convex subset of $\bbR^n$ containing $\vec{x}_1$, $\vec{x}_2$, $\dots$, $\vec{x}_k$.
In particular, we write $[\vec{x}_1, \vec{x}_2] =\conv(\vec{x}_1, \vec{x}_2)$ for the closed \emph{segment} with endpoints $\vec{x}_1$ and $\vec{x}_2$.
The \emph{linear span} of $\vec{x}_1$, $\vec{x}_2$, $\dots$, $\vec{x}_k$ is the set $\linspan(\vec{x}_1, \vec{x}_2, \dots, \vec{x}_k) = \set{ \sum_{m = 1}^k \alpha_m \vec{x}_m : \vec{\alpha} \in \bbR^k}$, which is the smallest $\bbR$-linear subspace of $\bbR^n$ containing $\vec{x}_1$, $\vec{x}_2$, $\dots$, $\vec{x}_k$.

The \emph{support} of a vector $\vec{x} \in \bbR^{n}$ is the set $\supp(\vec{x}) = \set{i \in \intupto{n} : x_i \neq 0}$.
The \emph{standard simplex} in $\bbR^n$ is the set $\Delta_n = \set{ \vec{x}\in \bbR^{n} : \vec{1}^{\top} \vec{x} = 1,\, \vec{x}\geq \vec{0} }$.
For a non-empty $S \subseteq \intupto{n}$, the \emph{face} of $\Delta_n$ associated with $S$ is the set
\begin{equation*}
    \Delta_n(S) = \set{\vec{x} \in \Delta_n : \supp(\vec{x}) \subseteq S},
\end{equation*}
and we also define
\begin{align*}
    \relint{\Delta_n(S)} &= \set{\vec{x} \in \Delta_n : \supp(\vec{x}) = S},\\
    \relbdy{\Delta_n(S)}&= \Delta_n(S) \setminus \relint{\Delta_n(S)},    
\end{align*}
where $\relint{\Delta_n(S)}$ and $\relbdy{\Delta_n(S)}$ are the \emph{relative interior} of $\Delta_n(S)$ and the \emph{relative boundary} of $\Delta_n(S)$ respectively.
In this notation, $\Delta_n(\intupto{n})$ is an alias for $\Delta_n$.

As for the graph-related notation, $G = (V, E)$ denotes in the sequel an unweighted undirected graph with no loops on the set $V$ and with set of edges $E \subseteq \binom{V}{2}$. Given two distinct nodes $i$, $j \in V$, we call $i$ a \emph{neighbor} of $j$, and write $i \sim j$, if $\set{i , j } \in E$. A node $i\in V$ is an \emph{end} of an edge $e \in E$ if $i \in e$.
Given a non-empty $S \subseteq V$, we write $G[S]$ for the subgraph of $G$ \emph{induced by $S$}, i.e., $G[S]$ is the graph on the set $S$ in which two nodes $i$, $j \in S$ are adjacent if and only if $\set{i, j} \in E$.
A non-empty $C \subseteq V$ is called a \emph{clique} if $i \sim j$ for every pair of distinct $i$, $j \in C$. A graph is said \emph{complete} if its set of nodes is a clique. 
The \emph{degree} of a node in a graph is the number of neighbors that node has in the graph. A graph $G'$ is \emph{regular} if every node of $G'$ has the same degree in $G'$, and, in particular, it is called a $d$-regular graph if every node of $G'$ has degree equal to $d$ in $G'$.
We also recall that an \emph{automorphism} of a graph is an isomorphism with itself, i.e., a permutation $\sigma$ of its nodes such that two nodes $i$ and $j$ are neighbors if and only if also $\sigma(i)$ and $\sigma(j)$ are neighbors.

\section{Parametric Motzkin-Straus Programs}\label{sec:param_ms}
Consider a graph $G = (V, E)$ on a finite non-empty set $V$, and let $n =  \card{V}$. Without loss of generality, we assume that $V = \intupto{n}$ to simplify the notation.

Let $\mtx{A}=[a_{ij}]_{i,j}$ be the adjacency matrix of $G$, i.e., the $n \times n$ symmetric matrix with coefficient $a_{ij}$ equal to $1$ if $i \sim j$ and equal to $0$ otherwise, and write $\mtx{I}$ for the $n \times n$ identity matrix. Fix now a parameter $c \in \bbR$ and consider the QP
\begin{maxi}|l|
    {\vec{x} \in \Delta_n}
    {f_c(\vec{x}) =\QuadForm{(\mtx{A} + c\mtx{I})}{\vec{x}}}
    { \label{prog:ms-c}}
    { }
\end{maxi}
and the associated  Lagrangian \cite{luenberger_linear_2016}
\begin{equation*}
    \Lagr (\vec{x},  \mu_0, \vec{\mu}) = 
    f_c(\vec{x})
    + \mu_0 (\vec{1}^\top \vec{x} - 1)
    + \vec{\mu}^\top \vec{x},
\end{equation*}
where $(\vec{x}, \mu_0, \vec{\mu} ) \in \bbR^n \times \bbR \times \bbR^n$.
Program~\eqref{prog:ms-c} is discussed in \cite{bomze_1997} for $c = \frac{1}{2}$ and in \cite{bomze_annealed_2002} in its more general formulation. Note that Program~\eqref{prog:ms-c} is precisely the Motzkin-Straus QP in case $c = 0$.

\begin{definition}\label{def:kkt}
    Given $\vec{x} \in \bbR^n$, we call $\vec{x}$ a \emph{KKT point} for Program~\eqref{prog:ms-c}, and write $\vec{x} \in \kkt{c}$, if $\vec{x} \in \Delta_n$ and some $(\mu_0, \vec{\mu}) \in \bbR \times \bbR^n$ exists such that
    \begin{subequations}\label{eq:kkt-sys}
        \begin{empheq}[left={\empheqlbrace}]{alignat = 2}
                    &\dfrac{\partial \Lagr}{ \partial \vec{x}} (\vec{x}, \mu_0, \vec{\mu}) = \vec{0}, & & \label{eq:kkt:lagr}\\
                    &\mu_i x_i = 0 &\qquad &\textrm{for every  $i \in V$},\label{eq:kkt:support}\\
                    &\vec{\mu} \geq \vec{0}. & &\label{eq:kkt:sign}
        \end{empheq}
    \end{subequations}
    We call $\vec{x}$ a \emph{generalized}
   \emph{KKT} (gKKT) \emph{point} for Program~\eqref{prog:ms-c}, and we write $\vec{x} \in \gkkt{c}$, if $\vec{x} \in \Delta_n$ and some $(\mu_0, \vec{\mu}) \in \bbR \times \bbR^n$ exists such that Eq.~\eqref{eq:kkt:lagr} and Eq.~\eqref{eq:kkt:support} hold. 
\end{definition}
In other words, the notion of KKT point is generalized%
\footnote{
    Generalized KKT points appear also in \cite{bomze_1997,bomze_annealed_2002}, where $\vec{x} \in \Delta_n$  is called a generalized KKT point if some $(\mu_0, \vec{\mu}) \in \bbR \times \bbR^n$ exists such that both Eq.~\eqref{eq:kkt:lagr} and $\vec{\mu}^{\top} \Vec{x} = 0$ are satisfied.
    If Eq.~\eqref{eq:kkt:sign} is satisfied, then $\vec{\mu}^{\top} \Vec{x} = 0$ and  Eq.~\eqref{eq:kkt:support} are equivalent, but in general the constraint $\vec{\mu}^{\top} \Vec{x} = 0$ is weaker than Eq.~\eqref{eq:kkt:support}. }
in  Definition~\ref{def:kkt} by dropping Eq.~\eqref{eq:kkt:sign}, which prescribes the sign condition on $\vec{\mu}$ appearing in Eq.~\eqref{eq:kkt-sys}. This obviously entails that $\kkt{c} \subseteq \gkkt{c}$.

A brief comment on this sign condition is helpful to provide insight into the definition of generalized KKT points. As it is well known \cite{luenberger_linear_2016}, KKT conditions: i) cause every multiplier associated with inactive constraints to be equal to zero; ii) prescribe the sign for multipliers associated with inequality constraints. In our case, for every $i \in V$ the multiplier $\mu_i$ appearing in Eq.~\eqref{eq:kkt-sys} corresponds to the constraint $x_i \geq 0$, and so the sign condition $\mu_i \geq 0$ appears in Eq.~\eqref{eq:kkt-sys}, together with $\mu_i x_i = 0$, which is meant to force $\mu_i = 0$ in case $x_i \neq 0$.
Consequently, for a given point $\vec{x}$ in the feasible set, dropping the sign condition on all multipliers is precisely how KKT conditions for $\vec{x}$ are modified if active inequality constraints are replaced with equality constraints.
Following this idea, we can prove that every element of $\gkkt{c}$ is a KKT point for a suitable modification of Program~\eqref{prog:ms-c}.

\begin{proposition}\label{prop:gkkt-prog}
   For every non-empty $S \subseteq V$, denote by $K(S)$ the set of KKT points for the program
    \begin{maxi}|l|[0]
            { \vec{x} \in \relint{\Delta_n(S)}}
            {f_c(\vec{x}).}
            {\label{p:face}}
            { }
    \end{maxi}
    Then
    \begin{equation}\label{eq:bigcup}
        \gkkt{c} = \bigcup_{S \in 2^V \setminus  \set{\emptyset}} K(S).
    \end{equation}
\end{proposition}
\begin{proof}
    By construction $K(S) \subseteq \relint{\Delta_n(S)}$ for every $S \in 2^V \setminus  \set{\emptyset}$, and 
    note that $ \set{ \relint{\Delta_n(S)} : S \in 2^V \setminus  \set{\emptyset}}$ is a partition for $\Delta_n$, hence Eq.~\eqref{eq:bigcup} is equivalent to
    \begin{equation}\label{eq:bigcup-equiv}
        \gkkt{c} \cap \relint{\Delta_n(S)}   = K(S) \qquad \text{for every $S \in 2^V \setminus  \set{\emptyset}$.}
    \end{equation}
    Let $S \in 2^V \setminus  \set{\emptyset}$ and let $\hat{\vec{x}} \in \relint{\Delta_n(S)}$, i.e., let $\hat{\vec{x}} \in \Delta_n$ satisfy $\supp(\hat{\vec{x}}) = S$.
    Then  $\hat{\vec{x}} \in \gkkt{c}$ if and only if some $(\mu_0, \vec{\mu}) \in  \bbR \times \bbR^n$ exists such that
    \begin{subequations}\label{eq:kkt-sys-II}
        \begin{empheq}[left={\empheqlbrace}]{alignat = 2}
            &\dfrac{\partial \Lagr}{ \partial \vec{x}} (\hat{\vec{x}}, \mu_0, \vec{\mu}) = \vec{0}, & &\notag \\
            &\mu_i= 0 &\qquad &\text{for every  $i \in S$}. \label{eq:kkt-sys-II:support}
        \end{empheq}
    \end{subequations}
    Note that Program~\eqref{p:face} can be written as
\begin{maxi!}|l|[0]
    {\vec{x} \in \bbR^n} 
    {f_c(\vec{x}) \notag} 
    { \label{p:face-equiv}} 
    { } 
    \addConstraint{ \vec{1}^{\top} \vec{x}}{= 1}{ \notag}
    \addConstraint{ x_i }{> 0}{\qquad \text{for every $i \in S$} \label{eq:face-supp}}
    \addConstraint{ x_i }{= 0}{\qquad \text{for every $i \in V \setminus S$}, \label{eq:face-nosupp}}
\end{maxi!}
and  $\hat{\vec{x}}$ is a KKT point for Program~\eqref{p:face-equiv} if and only if some $(\mu_0, \vec{\mu}) \in  \bbR \times \bbR^n$ exists such that
    \begin{subequations}\label{eq:face-kkt-sys}
        \begin{empheq}[left={\empheqlbrace}\;]{alignat = 2}
            &\dfrac{\partial \Lagr}{ \partial \vec{x}} (\hat{\vec{x}}, \mu_0, \vec{\mu}) = \vec{0}, & &\notag \\
            &\mu_i \hat{x}_i= 0 &\qquad &\text{for every  $i \in V$}, \label{eq:face:support}\\
            &\mu_i \geq 0 &\qquad &\text{for every  $i \in S$}. \label{eq:face:sign}
        \end{empheq}
    \end{subequations}
The constraints given in Eq.~\eqref{eq:face-supp} are not active in $\vec{x} = \hat{\vec{x}}$, hence Eq.~\eqref{eq:kkt-sys-II:support} holds if and only if both Eq.~\eqref{eq:face:support} and Eq.~\eqref{eq:face:sign} hold simultaneously. Consequently,
Eq.~\eqref{eq:kkt-sys-II} is equivalent to Eq.~\eqref{eq:face-kkt-sys}, which means that $\hat{\vec{x}} \in \gkkt{c}$ if and only if $\hat{\vec{x}}$ is a KKT point for Program~\eqref{p:face}. Since $S$ is arbitrary in $ 2^V \setminus  \set{\emptyset}$, this proves Eq.~\eqref{eq:bigcup-equiv}.
\end{proof}
\noindent
We remark that $\kkt{c} \cap \relint{\Delta_n} = \gkkt{c} \cap \relint{\Delta_n}$ by Proposition~\ref{prop:gkkt-prog}.

To prove that Definition~\ref{def:kkt} is satisfied, it is helpful to use the following simpler description of $\kkt{c}$ and $\gkkt{c}$, as in \cite{gibbons_continuous_1997,bomze_annealed_2002,bomzeRegularityDegeneracyDynamics2002}:

\begin{proposition}\label{prop:kkt}
     Let $\vec{x} \in \Delta_n$, let $\mtx{M} \in \bbR^{n \times n}$, define $\lambda = \QuadForm{\mtx{M}}{\vec{x}}$ and consider the properties:
    \begin{enumerate}[label = \textup{(S\arabic*)},left = 0pt]
        \item $(\mtx{M}\vec{x} )_i = (\mtx{M}\vec{x} )_j$ for every $i$, $j \in \supp(\vec{x})$; \label{item:support-weak}
        \item $(\mtx{M}\vec{x} )_i = \lambda$ for every $i \in \supp(\vec{x})$; \label{item:support}
        \item $(\mtx{M}\vec{x})_i \leq \lambda$ for every $i \in V \setminus \supp(\vec{x})$.\label{item:not-support}
    \end{enumerate}
    Then:
    \begin{mylist}
        \item  \ref{item:support-weak} and \ref{item:support} are equivalent; \label{item:mild}
        \item   $\vec{x} \in \kkt{c}$ if and only if both \ref{item:support} and \ref{item:not-support} hold for $\mtx{M} = \mtx{A} + c \mtx{I}$;
        \label{item:prop:kkt}
        \item   $\vec{x} \in \gkkt{c}$ if and only if \ref{item:support} holds for $\mtx{M} = \mtx{A} + c \mtx{I}$.\label{item:prop:gkkt}
   \end{mylist}
\end{proposition}
\begin{proof}
    \ref{item:mild}: First, observe that
    \begin{equation}
        \QuadForm{\mtx{M}}{\vec{x}} = \sum_{i \in \supp(\vec{x})} x_i (\mtx{M}\vec{x})_i.\label{eq:fc-equiv-supp}
    \end{equation}
    Clearly \ref{item:support} implies \ref{item:support-weak}. Conversely, if \ref{item:support-weak} holds, then for some $r \in \bbR$ we may write $(\mtx{M}\vec{x} )_i = r$ for every $i\in \supp(\vec{x})$, and Eq. \eqref{eq:fc-equiv-supp} yields
    \begin{equation*}
        \lambda = f_c(\vec{x}) = \sum_{i \in \supp(\vec{x})} x_i (\mtx{M}\vec{x})_i = \sum_{i \in \supp(\vec{x})} x_i r = r,
    \end{equation*}
    hence \ref{item:support} holds.
    
    \ref{item:prop:kkt}, \ref{item:prop:gkkt}: Assume $\mtx{M} = \mtx{A} + c \mtx{I}$ and note that Eq.~\eqref{eq:kkt:lagr} is equivalent to  $\vec{\mu}  =  - (2\mtx{M}\vec{x} + \mu_0 \Vec{1})$, hence Eq.~\eqref{eq:kkt:support} is equivalent to 
    \begin{equation*}
        (\mtx{M}\vec{x})_i  = -\frac{\mu_0}{2} \qquad \text{for every $i \in \supp(\vec{x})$,}\label{eq:lagrangian-deriv}
    \end{equation*}
    whereas Eq.~\eqref{eq:kkt:sign} is equivalent to 
    \begin{equation*}
        (\mtx{M}\vec{x})_i  \leq -\frac{\mu_0}{2} \qquad \text{for every $i \in V$}.
    \end{equation*}
    Then \ref{item:prop:kkt} and \ref{item:prop:gkkt} follow by \ref{item:mild}.
\end{proof}
Proposition~\eqref{prop:kkt} enables us to introduce a notation that extends Definition~\ref{def:kkt}.

\begin{definition}\label{def:alternative}
    Consider a graph $G'$ with set of nodes $V$. Then we write $\mtx{A}_{G'}$ for the adjacency matrix of $G'$, and we define for every $\gamma \in \bbR$
    \begin{align*}
        \kktsubgraph{G'}{\gamma} &=\set{\vec{x}\in \Delta_n : \text{\ref{item:support} and \ref{item:not-support} hold for $\mtx{M} = \mtx{A}_{G'} + \gamma \mtx{I}$}},\\
        \gkktsubgraph{G'}{\gamma} &=\set{\vec{x}\in \Delta_n : \text{\ref{item:support} holds for $\mtx{M} = \mtx{A}_{G'} + \gamma \mtx{I}$}}.
    \end{align*}    
\end{definition}
By Proposition~\ref{prop:kkt}, it is immediate to check that $\kktsubgraph{G'}{\gamma}$ (resp. $\gkktsubgraph{G'}{\gamma}$) is the set of KKT (resp. generalized KKT) points for the program
   \begin{maxi}|l|
        {\vec{x} \in \Delta_n}
        {g(\gamma,G', \vec{x}) =  \QuadForm{\left(\mtx{A}_{G'} + \gamma\mtx{I}\right)}{\vec{x}}.}
        { \label{prog:ms-c-compl}}
        { }
    \end{maxi}
\begin{theorem}\label{thm:tool}
    Let $S$ be a non-empty subset of $V$, and consider two graphs $G'$, $G''$ with set of nodes $V$ such that  $G'[S] = G''[S]$.
    Let $\vec{x} \in \Delta_n$ with $S = \supp(\vec{x})$.
    \begin{mylist}
        \item If $\vec{x} \in \gkktsubgraph{G'}{c}$, then $\vec{x} \in \gkktsubgraph{G''}{c}$; \label{i:tool:gkkt}
        \item If $\vec{x} \in \kktsubgraph{G'}{c}$, then it is not necessarily true that $\vec{x} \in \kktsubgraph{G''}{c}$.\label{i:tool:kkt}
        \item Assume $c \geq 0$ and that in $G''$ every node in $V\setminus S$ has no neighbors. If $\vec{x} \in \gkktsubgraph{G'}{c}$ then $\vec{x} \in \kktsubgraph{G''}{c}$. \label{i:tool:upgrade}
    \end{mylist}
\end{theorem}
\begin{proof}
    \ref{i:tool:gkkt}: Assume  $\vec{x} \in \gkktsubgraph{G'}{c}$. A direct computation shows that $G'[S]= G''[S]$ entails $(\mtx{A}_{G'} \vec{x})_i= (\mtx{A}_{G''}\vec{x})_i$ for every $i \in S$, and so the thesis follows by Definition~\ref{def:alternative}.
    
    \ref{i:tool:kkt}: Let $V = \set{1, 2, 3}$, let $S = \set{1, 2}$ and suppose $G'$ and $G''$ satisfy
    \begin{equation*}
        \mtx{A}_{G'} =
        \begin{pmatrix}
            0 & 1 & 0\\
            1 & 0 & 0\\
            0 & 0 & 0
        \end{pmatrix}, \qquad \mtx{A}_{G''}=
        \begin{pmatrix}
            0 & 1 & 1\\
            1 & 0 & 1\\
            1 & 1 & 0
        \end{pmatrix}.
    \end{equation*}
    Let $\vec{x} = \big(\frac{1}{2}, \frac{1}{2}, 0\big)^{\top}$.
    Then \begin{equation*}
        \mtx{A}_{G'} \vec{x} + c\vec{x}=\dfrac{1}{2}
        \begin{pmatrix}
            c +1 \\
            c + 1\\
            0
        \end{pmatrix} \qquad \mtx{A}_{G''} \vec{x} + c\vec{x} =\dfrac{1}{2}
        \begin{pmatrix}
            c +1 \\
            c + 1\\
            2
        \end{pmatrix},
        \end{equation*}
        hence $\vec{x}\in \kktsubgraph{G'}{c} \setminus\kktsubgraph{G''}{c}$ in case $-1 \leq c < 1$.
        
    \ref{i:tool:upgrade} Assume  $\vec{x} \in \gkktsubgraph{G'}{c}$. The condition on $c$ gives $(\mtx{A}_{G'} \vec{x})_i= (\mtx{A}_{G''}\vec{x})_i \geq 0$ for every $i \in S$, whereas $(\mtx{A}_{G''}\vec{x})_i = 0$ for every $i \in V \setminus S$, hence $\vec{x} \in \kktsubgraph{G''}{c}$.
\end{proof}
Theorem~\ref{thm:tool} offers an early indication that elements of $\gkkt{c}$ reflect the properties of induced subgraphs of $G$, and that these elements are also in $\kkt{c}$ only if additional conditions depending on the overall structure of $G$ are met. See \cite{gibbons_continuous_1997,bomzeRegularityDegeneracyDynamics2002,tang_maxima_2022} for additional structural information related to local optima of the Motzkin-Straus QP.

It is possible to describe how $\kkt{c}$ and $\gkkt{c}$ are affected by replacing the graph $G$ appearing in Program~\eqref{prog:ms-c} with the complement graph of $G$. 
This kind of replacement was considered also in \cite{motzkin_straus_1965,pardalos_global_1990,bomze_1997}, and is related to the well-known property that a maximum clique for a given graph is, for its complement graph, a maximum independent set \cite{diestel_2005_graph_theory}. Let $\overline{G}$ denote the complement graph of $G$.
\begin{proposition}\label{prop:complement}
    Consider the \emph{minimization} program
    \begin{mini}|l|
        {\vec{x} \in \Delta_n}
        {\QuadForm{[\mtx{A} + (1 - c)\mtx{I}]}{\vec{x}}.} 
        {\label{prog:ms-c-min}} 
        { }
    \end{mini}
\begin{mylist}
    \item \label{i:complement:kkt} $\kktsubgraph{\overline{G}}{c}$ coincides
    with the set of KKT points for Program~\eqref{prog:ms-c-min};
    \item \label{i:complement:gkkt} $\gkktsubgraph{\overline{G}}{c} = \gkktsubgraph{G}{1 - c}$.
\end{mylist}
\end{proposition}
\begin{proof}
    \ref{i:complement:gkkt}: Recalling that $\mtx{A} = \mtx{A}_G$, let $\mtx{B} = \mtx{A}_{\overline{G}}$ and let $\vec{x} \in \Delta_n$. The identity $\mtx{A} + \mtx{B} + \mtx{I} = \vec{1}\vec{1}\!^{\top}$
    implies that
    \begin{equation*}
        \mtx{B}\vec{x} + c \vec{x} =  \vec{1} - \mtx{A}\vec{x} + (c - 1) \vec{x},
    \end{equation*}
    and so 
    \begin{equation}\label{eq:swap}
        (\mtx{B} \vec{x} + c \vec{x})_i = 1 - (\mtx{A}\vec{x} + (1 -c ) \vec{x})_i \qquad \text{for every } i \in V.
    \end{equation}      
    Consequently, $\vec{x} \in \gkktsubgraph{G}{1 - c}$ if and only if $\vec{x} \in \gkktsubgraph{\overline{G}}{c}$.
    
    \ref{i:complement:kkt}: Note that $\vec{x} \in \kktsubgraph{\overline{G}}{c}$ if and only if for some $\lambda \in \bbR$
    \begin{equation*}
        \begin{cases}
            (\mtx{B} \vec{x} + c \vec{x})_i = \lambda       & \quad\text{for every  $i \in \supp(\vec{x})$}, \\
            (\mtx{B} \vec{x} + c \vec{x})_i \leq \lambda    & \quad\text{for every  $i \in V \setminus \supp(\vec{x})$,}
        \end{cases}
    \end{equation*}
    i.e., by Eq.~\eqref{eq:swap}, if and only if for some $r \in \bbR$
    \begin{equation}\label{eq:kkt-min}
        \begin{cases}
            (\mtx{A}\vec{x} + (1 -c ) \vec{x})_i= r     & \quad\text{for every  $i \in \supp(\vec{x})$}, \\
            (\mtx{A}\vec{x} + (1 -c ) \vec{x})_i\geq r  & \quad\text{for every  $i \in V \setminus \supp(\vec{x})$,}
        \end{cases}
    \end{equation}
    and note that Eq.~\eqref{eq:kkt-min} is equivalent to the KKT conditions for Program~\eqref{prog:ms-c-min}.
\end{proof}
Proposition~\ref{prop:complement} has been inspired by \cite{motzkin_straus_1965}, in which a similar idea is used to show that the Motzkin-Straus QP is equivalent to a certain minimization program.

We remark that Proposition~\ref{prop:complement}.\ref{i:complement:gkkt} applied for $c = \frac{1}{2}$ gives
\begin{equation*}
    \gkktsubgraph{\overline{G}}{\frac{1}{2}} = \gkktsubgraph{G}{\frac{1}{2}},
\end{equation*}
a curious equality involving on both sides the QP studied in \cite{bomze_1997}.
\section{Characteristic Vectors}\label{sec:char_vec}
\begin{definition}
    Let $S$ be a non-empty subset of $V$. We call \emph{characteristic vector}%
        \footnote{Some authors adopted a different terminology. For instance, the vector that we denote by $\vec{x}^{S}$ is called in \cite{bomze_1997} the barycenter of the face of $\Delta_n$ corresponding to $S$, a terminology providing more geometrical insight.}
    representing $S$ in $\Delta_n$ the vector $\vec{x}^S \in \Delta_n$ defined by
    \begin{equation*}
        x_i^S =
        \begin{cases}
          \card{S}^{-1}     & \quad\text{if $i \in S$},\\
          0                 & \quad\text{otherwise}.
        \end{cases}
    \end{equation*}
\end{definition}
Characteristic vectors representing maximum (resp. maximal) cliques emerge as global (resp. local) solutions to Program~\eqref{prog:ms}, see e.g. \cite{motzkin_straus_1965,pelillo_jagota_1996,tang_maxima_2022}. Indeed, the injection mentioned in Sect.~\ref{sec:intro} maps every non-empty $S \subseteq V$ to its corresponding characteristic vector. We remark that $\vec{x}^S$ may be in $\kkt{c}$ even if $S$ is not a clique, i.e., even if the graph $G[S]$ is not complete. What is known, however, is that if $\vec{x}^S$ is in $\gkkt{c}$, then the graph $G[S]$ is necessarily regular, and this is independent of the value of $c$.
\begin{proposition}\label{prop:regular-bomze}
    Let $\vec{x}$ be a characteristic vector. Then $\vec{x}\in \gkkt{c}$ if and only if $G[\supp(\vec{x})]$ is regular.
\end{proposition}
\begin{proof}
    Set $S = \supp(\vec{x})$, so that $\vec{x} = \vec{x}^{S}$.
    For every $i \in S$, call $d_i$ be the number of neighbors of $i$ in $S$. Then $(\mtx{A} \vec{x}^S)_i = d_i  \card{S}^{-1}$, thus $((\mtx{A} + c \mtx{I})\vec{x}^S)_i = (d_i + c) \card{S}^{-1}$. By Proposition~\ref{prop:kkt}, the vector $\vec{x}^S$ is in $\gkkt{c}$ if and only if for some $\lambda \in \bbR$ the equality $(d_i + c)  \card{S}^{-1} = \lambda$ holds for every $i \in S$. This is possible if and only if $d_i$ has the same value for every $i \in S$, i.e., if and only if $G[S]$ is regular.  
\end{proof}
Proposition~\ref{prop:regular-bomze} is proved in \cite{bomze_1997} for the case $c = \frac{1}{2}$, and in \cite{bomze_annealed_2002} it is mentioned that the same proof works for $0 \leq c \leq 1$.
We can say more in the following generalization of \cite[Proposition~6]{bomze_annealed_2002}:
\begin{proposition}\label{prop:two-entail-all}
    Let $\vec{x} \in \Delta_n$ and suppose two distinct $c_1$, $c_2 \in \bbR$ exist such that 
    $\vec{x} \in \gkkt{c_j}$ for $j = 1$, $2$. Then $\vec{x} = \vec{x}^{S}$ for some $S \subseteq V$ and $G[S]$ is regular.
\end{proposition}
\begin{proof}
    Set $S = \supp(\vec{x})$. By Proposition~\ref{prop:kkt}, there exist $\lambda_1$, $\lambda_2 \in \bbR$ such that for $j = 1$, $2$, the equality $((\mtx{A} + c_{j} \mtx{I})\vec{x})_i = \lambda_j$ holds for every $i \in S$. Therefore, \emph{every} non-zero component of $\vec{x}$ equals $(\lambda_1 - \lambda_2)(c_1 - c_2)^{-1}$. This entails $\vec{x} = \vec{x}^S$, hence $G[S]$ is regular by Proposition~\ref{prop:regular-bomze}.
\end{proof}

By Proposition~\ref{prop:two-entail-all}, a characteristic vector is in $\gkkt{\gamma}$ either for every value of $\gamma \in \bbR$ or for no value of $\gamma \in \bbR$.
By contrast, a rather different behavior can be observed for elements of $\Delta_n$ that are not characteristic vectors, and the general element of $\Delta_n$ is of this type, considering that exactly $2^{n} - 1$ elements of $\Delta_n$ are characteristic vectors.
\begin{proposition}\label{prop:one-at-most}
    Let $\vec{x} \in \Delta_n$ and suppose $\vec{x}$ is \emph{not} a characteristic vector. Then $\vec{x} \in \gkkt{\gamma}$ for  at most \emph{one} value of $\gamma \in \bbR$.
\end{proposition}
\begin{proof}
    If there exist distinct $c_1$, $c_2 \in \bbR$ such that $\vec{x} \in \gkkt{c_j}$ for $j = 1$, $2$, then $\vec{x}$ is a characteristic vector by Proposition~\ref{prop:two-entail-all}, and this contradicts the hypothesis on $\vec{x}$.
\end{proof}
\section{Automorphisms of Induced Subgraphs}\label{sec:automorph}
The elements of $\kkt{c}$ need not be characteristic vectors. For instance, suppose $G$ is the graph on the set of nodes $\set{1,2,3}$ and edges $\set{\set{1,3}, \set{2,3}}$, sometimes called the \emph{cherry graph} (Fig.~\ref{fig:cherry}). It is easy to check that the point $\tilde{\vec{x}} = \big(\frac{1}{4}, \frac{1}{4}, \frac{1}{2}\big)^{\top}$ --- clearly not a characteristic vector --- is an element of $\kkt{0}$, as well as an instance of spurious solution for the Motzkin-Straus QP \cite{pardalos_global_1990}.

\begin{figure}[t] 
\centering
\includegraphics{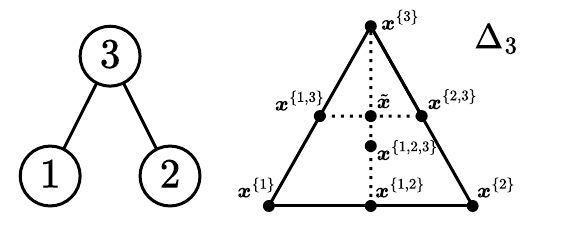} 
\caption{The cherry graph, characteristic vectors representing non-empty subsets of $\set{1,2,3}$ in $\Delta_3$, and the spurious solution $\tilde{\vec{x}}$}\label{fig:cherry}
\end{figure}

To introduce our next result, note that both the vector $\tilde{\vec{x}}$ and the cherry graph are preserved if node $1$ and node $2$ are exchanged. To be rigorous, call $\sigma$ the permutation on $\set{ 1, 2, 3 }$ swapping $1$ and $2$. Then $\sigma$ is an automorphism for the cherry graph and at the same time $\tilde{\vec{x}}$ is invariant under the pull-back by $\sigma$, i.e., the vector $\tilde{\vec{x}}$ is preserved if its $i$-th coordinate is replaced with its $\sigma(i)$-th coordinate for every $i \in V$. Indeed, this is an instance of a more general fact.

\begin{lemma}\label{lem:symmetrize}
    Let $\vec{x} \in \gkkt{c}$, set $S = \supp(\vec{x})$ and let $\mathcal{G}$ be a group of automorphisms for the induced subgraph $G[S]$. Then there exists a point
    $\hat{\vec{x}} \in \gkkt{c}$ such that $\supp(\hat{\vec{x}}) = S$ and such that $\hat{x}_{\sigma(i)} = \hat{x}_i$ for every $i \in S$ and every $\sigma \in \mathcal{G}$.
\end{lemma}
\begin{proof}
    For every $\sigma \in \mathcal{G}$, denote by $\sigma^* \vec{x}$ the vector in $\Delta_n$ defined by%
        \footnote{The reader may note a subtle abuse of notation for the pull-back: we are identifying $\sigma$, which is a permutation on $S$, and the permutation on $V$ extending $\sigma$ to $V$ so that it keeps fixed every node in $V \setminus S$.}
    \begin{equation*}
        (\sigma^* \vec{x})_i =
        \begin{cases}
        x_{\sigma(i)} &\quad\text{if $i \in S$},\\
        0 &\quad\text{otherwise.}
        \end{cases}
    \end{equation*}
    It suffices to prove that $\hat{\vec{x}} = \card{\mathcal{G}}^{-1} \sum_{\sigma \in \mathcal{G}} \sigma^*\vec{x}$ satisfies the desired properties. Note that $\hat{\vec{x}} \in \Delta_n$ by convexity of $\Delta_n$, and that $\supp(\sigma^* \vec{x})= S$ for every $\sigma \in \mathcal{G}$, thus $\supp(\hat{\vec{x}}) = S$ by construction.
    By hypothesis on $\vec{x}$, some $\lambda$ exists such that $((\mtx{A} + c \mtx{I}) \vec{x})_i = \lambda$ for every $i \in S$. Then, using that $a_{ij} = a_{\sigma(i) \sigma(j)}$ holds for every $\sigma \in \mathcal{G}$ and every $i$, $j \in S$, we get
    \begin{align*}
        ((\mtx{A} + c \mtx{I}) \sigma^* \vec{x})_i
        &= \Big( \sum_{j \in S} a_{ij} x_{\sigma(j)}\Big) + c x_{\sigma(i)}\\
        &= \Big(\sum_{j \in S} a_{\sigma(i)\sigma(j)} x_{\sigma(j)}\Big) + c x_{\sigma(i)}\\
        &= ((\mtx{A} + c \mtx{I}) \vec{x})_{\sigma(i)} = \lambda,
    \end{align*}
    and so
    \begin{equation*}
        ((\mtx{A} + c \mtx{I}) \hat{\vec{x}})_i
        = \dfrac{1}{\card{\mathcal{G}}} \sum_{\sigma\in \mathcal{G}} 
                ((\mtx{A} + c \mtx{I}) \sigma^* \vec{x})_i
                = \lambda,
    \end{equation*}
    showing that $\hat{\vec{x}} \in \gkkt{c}$. Finally, let $\tau \in \mathcal{G}$. By hypothesis $\mathcal{G}$ is a group, hence $\mathcal{G} \tau =  \mathcal{G}$, therefore
    \begin{equation*}        
    \tau^{*} \hat{\vec{x}} 
    = \dfrac{1}{\card{\mathcal{G}}} \sum_{\sigma\in \mathcal{G}} \tau^{*}( \sigma^{*} \hat{\vec{x}})
    = \dfrac{1}{\card{\mathcal{G}}} \sum_{\sigma\in\mathcal{G}} (\sigma \tau)^{*}\hat{\vec{x}}
    = \dfrac{1}{\card{\mathcal{G}}} \sum_{\sigma\in \mathcal{G}}
    \sigma^{*} \hat{\vec{x}}
    = \hat{\vec{x}}.
    \end{equation*} 
\end{proof}
Lemma~\ref{lem:symmetrize} will be central in our next result, which relies also on the following definition:
\begin{definition}\label{def:induced-partition}
    Given $\vec{x} \in \Delta_n$, define on $\supp(\vec{x})$ the equivalence relation $\sim_{\vec{x}}$ such that
    \begin{equation*}
        i \sim_{\vec{x}} j \qquad \text{if and only if } \qquad x_i = x_j.
    \end{equation*}
    The \emph{partition induced by $\vec{x}$} is the family of the equivalence classes of $\sim_{\vec{x}}$.
\end{definition}

In the sequel, for a given non-empty $S \subseteq V$ we write $\mtx{A}[S, S]$ for the \emph{principal submatrix} of $\mtx{A}$ having entries in the rows and columns of $\mtx{A}$ indexed by $S$, see \cite{horn_johnson_2013_MA}.

\begin{theorem}\label{thm:nuovo_symmetry}
    Consider a non-empty $S \subseteq V$.
    Suppose there exists $\vec{x} \in \kkt{c}$ such that $S = \supp(\vec{x})$ and suppose  $-c$ is \emph{not} an eigenvalue of $\mtx{A}[S,S]$.
    Then every class of the partition induced by $\vec{x}$ is invariant under \emph{every} automorphism of $G[S]$.
\end{theorem}
\begin{proof}
    The thesis is that $x_{\sigma(i)} = x_i$ for every $i \in S$ and every automorphism $\sigma$ of $G[S]$.
    Let $\sigma$ be an automorphism of $G[S]$ and let $\mathcal{G}$ be the group of automorphisms generated by $\sigma$. By Lemma~\ref{lem:symmetrize}, there exists $\hat{\vec{x}} \in \gkkt{c}$ satisfying $\supp(\hat{\vec{x}}) = S$ and $\hat{x}_{\sigma(i)} = \hat{x}_i$ for every $i \in S$. 
    By hypothesis on $c$, there exists a unique $\vec{z} \in \bbR^n$ such that
    \begin{equation*}
        \begin{cases}
            (\mtx{A}\vec{z})_i + cz_i = 1 &\quad\text{for every }i \in S,\\
            z_i = 0 &\quad\text{for every }i \in V \setminus S,\\
        \end{cases}
    \end{equation*}
    and by Proposition~\ref{prop:kkt} every element of $\gkkt{c}$ with support equal to $S$ is necessarily a multiple of $\vec{z}$. Since $\vec{z}$ admits at most one multiple in $\Delta_n$, this means that $\vec{x} = \hat{\vec{x}}$.
\end{proof}

Given a non-empty $S \subseteq V$, Theorem~\ref{thm:nuovo_symmetry} may provide some information about the orbits of the automorphisms of $G[S]$.
Specifically, \emph{if} $-c$ is not an eigenvalue of $\mtx{A}[S,S]$ \emph{and} we are able to find some $\vec{x} \in \gkkt{c}$ with $\supp(\vec{x}) = S$, \emph{then} Theorem~\ref{thm:nuovo_symmetry} states that for every automorphism $\sigma$ of $G[S]$ the partition on $S$ given by $\set{\set{\sigma^k (i) : k \in \bbZ} : i \in S}$ --- namely, the collection of the orbits under $\sigma$ ---  is a refinement of the partition induced by $\vec{x}$.
For $\card{S} = 1$, $2$ note that $G[S]$ is regular, hence by Proposition~\ref{prop:regular-bomze} there exists $\vec{x} \in \gkkt{c}$ with $\supp(\vec{x}) = S$ regardless of the value of $c$.
However, for $\card{S} \geq 3$ the structure of $G[S]$ may be more complicated, and such an element of $\gkkt{c}$ may not exist.
\begin{proposition}\label{prop:support-c-0-1}
    Suppose three distinct nodes $i_1$, $i_2$, $i_3 \in V$ satisfy $i_1 \not \sim i_3$ and $i_2 \sim i_3$. Let $S' \subseteq V \setminus \set{i_1, i_2, i_3}$ be such every neighbor of $i_1$ in $S'$ is also a neighbor of $i_2$. Let $S = S' \cup \set{i_1, i_2, i_3}$.
    \begin{mylist}
        \item \label{it:cherry}If $i_1 \sim i_2$, then no element of $\gkkt{1}$ has support equal to $S$;
        \item \label{it:cherry-bar} If $i_1 \not \sim i_2$, then no element of $\gkkt{0}$ has support equal to $S$.
    \end{mylist}
\end{proposition}
\begin{proof}
    Let $\mtx{M}\in \bbR^{n \times n}$ such that $\mtx{M} = \mtx{A} + \gamma \mtx{I}$ for some $\gamma \in \bbR$.
    Write $\mtx{M} = [m_{ij}]_{i,j}$ and note that $m_{ij} = a_{ij}$ for every distinct $i$, $j \in S$, hence $m_{{i_2} j} - m_{{i_1} j} = a_{{i_2} j} - a_{{i_1} j} \geq 0$ for every $j \in S'$ by hypothesis on $S'$, and the inequality is strict for $j = i_3$, whereas
    \begin{equation*}
    m_{{i_2} i_1} - m_{{i_1} i_1} =  a_{i_1 i_2} - \gamma \qquad \text{and} \qquad m_{{i_2} i_2} - m_{{i_1} i_2} = \gamma - a_{i_1 i_2}.
    \end{equation*}
    \ref{it:cherry}: Suppose $\gamma = 1$ and $i_1 \sim i_2$. Then $a_{i_1 i_2}=1$ and $m_{{i_2} j} - m_{{i_1} j} = 0$ for $j = i_1$, $i_2$.
    Suppose now some $\vec{x} \in \gkkt{1}$ satisfies $\supp(\vec{x}) = S$. By  
    Proposition~\ref{prop:kkt}, every $i \in \supp(\vec{x})$ yields the same value for the expression $\sum_{j \in V} m_{ij} x_j$, hence: 
    \begin{equation*}
         0 
         =\sum_{j \in V} m_{i_2j} x_j - \sum_{j \in V} m_{i_1j} x_j
         =\sum_{j \in S} (m_{{i_2} j} - m_{{i_1 j}}) x_j \geq (m_{{i_2} i_3} - m_{{i_1} i_3 }) x_{i_3} >0, 
    \end{equation*}
    and this is absurd.

    \ref{it:cherry-bar}: This time, suppose $\gamma = 0$ and $i_1 \not\sim i_2$. As before $m_{{i_2} j} - m_{{i_1} j} = 0$ for $j = i_1$, $i_2$, and a vector $\vec{x} \in \gkkt{0}$ such that $\supp(\vec{x}) = S$ leads to a contradiction.
\end{proof}
The two conclusions of Proposition~\ref{prop:support-c-0-1} are indeed equivalent in light of Proposition~\ref{prop:complement}.
Observe that in Proposition~\ref{prop:support-c-0-1} the graph  $G[S]$ is isomorphic to either the cherry graph or its complement graph in case $S' = \emptyset$.

Even though it is possible that no element of $\gkkt{c}$ has support equal to $S$, the next proposition shows that a different scenario occurs if $\abs{c}$ is sufficiently large.

\begin{proposition}\label{prop:support-c-big}
    Let $S$ be a non-empty subset of $V$. There exists a bounded interval $I \subset \bbR$ such that if $c \notin I$ then at least an element of $\gkkt{c}$ has support equal to $S$.
\end{proposition}
\begin{proof}
    Call $s = \card{S}$ and assume $s \geq 3$, for otherwise the proof is trivial.
    Observe that
    \begin{equation*}
    0 < \min_{\vec{x} \in \Delta_n(S)} \QuadForm{}{\vec{x}} = \dfrac{1}{s} < \dfrac{1}{s-1} = \min_{\vec{x} \in \relbdy{\Delta_n(S)}} \QuadForm{}{\vec{x}},
    \end{equation*}
    thus as $\gamma \to -\infty$ we get the asymptotic estimates
    \begin{equation*}
    \max_{\vec{x} \in \Delta_n(S)} f_{\gamma}(\vec{x}) \sim \dfrac{\gamma}{s}, \qquad
    \max_{\vec{x} \in \relbdy{\Delta_n(S)}} f_{\gamma}(\vec{x}) \sim \dfrac{\gamma}{s-1}.
    \end{equation*}
    Consequently, if $c$ is negative and with modulus sufficiently large, then the function $f_c$ restricted to $\Delta_n(S)$ admits a maximum $\vec{z} \in \relint \Delta_n(S)$. By Proposition~\ref{prop:gkkt-prog}, it follows that $\vec{z} \in \gkkt{c}$.
    In case $c$ is positive, the same idea shows that if $1 - c$ is negative and with modulus sufficiently large, then there exists some $\vec{w} \in \gkktsubgraph{\overline{G}}{1-c}$ with $\supp(\vec{w}) =S$, and the proof is concluded after observing that $\vec{w} \in \gkkt{c}$ by Proposition~\ref{prop:complement}.
\end{proof}
A tighter bound on the interval $I$ in Proposition~\ref{prop:support-c-big} can be derived from the spectral radius of $\mtx{A}[S, S]$, as shown in \cite[Theorem~1]{pavan_pelillo_2003}.
\section{Barycentric Coordinates}\label{sec:kkt_convex}
Note that for a family of non-empty and pairwise disjoint subsets of $V$, the corresponding characteristic vectors are linearly independent.
This enables us to introduce a convenient representation of points of $\Delta_n$ based on barycentric coordinates, a type of coordinate system for simplices that is widely employed in finite element method and computer graphics, see e.g.  \cite{quarteroni_2017_NMfDP,hormann_2017_GBC}.

\begin{definition}\label{def:y}
    Consider a tuple $\tuplenum = (V_1, V_2, \dots, V_k)$ of non-empty and pairwise disjoint subsets of $V$. Given $\vec{x}\in \conv({\vec{x}}^{V_1}, {\vec{x}}^{V_2}, \dots, {\vec{x}}^{V_k})$, the \emph{vector of barycentric coordinates}%
    \footnote{We introduce here an abuse of language. Strictly speaking, the entries of $\bary_{\tuplenum} (\vec{x})$ constitute the barycentric coordinates of $\vec{z}$ with respect to the vectors $\vec{x}^{V_1}$,  ${\vec{x}}^{V_2}$, $\dots$, ${\vec{x}}^{V_k}$.}
    of $\vec{x}$ with respect to $\tuplenum$ is the unique 
    vector $\vec{y} = \bary_{\tuplenum} (\vec{x})$ in $\Delta_k$ such that $\vec{x} = \sum_{\ell = 1}^{k} y_{\ell} \vec{x}^{V_{\ell}}$.
\end{definition}
We remark that in the setting of Definition~\ref{def:y} we must have $x_i = x_j = y_{\ell}{\card{V_{\ell}}}^{-1}$ in case $i$, $j \in V_{\ell}$. A trivial example of barycentric coordinates is obtained considering $\tuplenum =( \set{1}, \set{2}, \dots, \set{n})$, for which $\bary_{\tuplenum}(\vec{x}) = \vec{x}$ for every $\vec{x} \in \Delta_n$.

Barycentric-coordinates notation will be combined with another concept that qualifies a particular way to partition the support of an element of $\Delta_n$. 
\begin{definition}\label{def:k-constant}
    Let $\vec{x}\in \Delta_n$. A partition $\mathcal{P}$ of $\supp(\vec{x})$ \emph{separates distinct values} of $\vec{x}$ if for every $i$, $j \in \supp(\vec{x})$ such that $x_i \neq x_j$ the nodes $i$ and $j$ belong to distinct classes of $\mathcal{P}$.
\end{definition}
In other words, a partition $\mathcal{P}$ of $\supp(\vec{x})$ separates distinct values of $\vec{x}$ if and only if $\mathcal{P}$ is a refinement of the partition induced by $\vec{x}$ described in Definition~\ref{def:induced-partition}.
Moreover, we remark that in the setting of Lemma~\ref{lem:symmetrize} the orbits of $\supp(\hat{\vec{x}})$ under the action of $\mathcal{G}$ constitute a partition of $\supp(\hat{\vec{x}})$ separating distinct values of $\hat{\vec{x}}$.

We now aim to discuss how a partition as in Definition~\ref{def:k-constant} is related to barycentric coordinates. This requires the following result:
\begin{proposition}\label{prop:hull}
    Consider a family $\set{V_1, V_2, \dots, V_k}$ of non-empty and pairwise disjoint subsets of $V$.
    Then
    \begin{equation*}
        \conv({\vec{x}}^{V_1}, {\vec{x}}^{V_2}, \dots, {\vec{x}}^{V_k}) 
        = \Delta_n \cap \linspan({\vec{x}}^{V_1}, {\vec{x}}^{V_2}, \dots, {\vec{x}}^{V_k}).
    \end{equation*}
\end{proposition}
\begin{proof}
    The vectors ${\vec{x}}^{V_1}, {\vec{x}}^{V_2}, \dots, {\vec{x}}^{V_k}$ are elements of $\linspan({\vec{x}}^{V_1}, {\vec{x}}^{V_2}, \dots, {\vec{x}}^{V_k} )$ and of $\Delta_n$, which are convex sets, hence $\conv({\vec{x}}^{V_1}, {\vec{x}}^{V_2}, \dots, {\vec{x}}^{V_k})$ is included in their intersection. The trivial inclusion $\conv({\vec{x}}^{V_1}, {\vec{x}}^{V_2}, \dots, {\vec{x}}^{V_k})
    \subseteq \Delta_n \cap \linspan({\vec{x}}^{V_1}, {\vec{x}}^{V_2}, \dots, {\vec{x}}^{V_k})$ is thus proved.
    
    To prove the reversed inclusion, consider some real coefficients $a_1$, $a_2$, $\dots$, $a_n$ such that  $\vec{x} = \sum_{\ell = 1}^{k} a_{\ell} {\vec{x}}^{V_\ell}$ is  an element of $\Delta_n$. For every $\ell \in \intupto{k}$, the hypotheses on $V_1$, $V_2$, $\dots$, $V_k$ entail the existence of some $i \in V_{\ell} \setminus \left(\cup_{m \neq \ell} V_m \right)$, and so $a_{\ell} \card{V_{\ell}}^{-1} = x_i \geq 0$, thus $a_{\ell} \geq 0$.
    Moreover, by $\vec{x}\in \Delta_n$ we obtain
    \begin{equation*}
        1 = \sum_{i = 1}^{n} x_i = \sum_{i = 1}^{n} \Big( \sum_{\ell = 1}^{k} a_{\ell} {\vec{x}}^{V_\ell}  {\Big)}_{\!i}
        = \sum_{\ell = 1}^{k} \sum_{i = 1}^{n} ( a_{\ell} {\vec{x}}^{V_\ell} )_i 
        = \sum_{\ell = 1}^{k} a_{\ell}.
    \end{equation*}
    But then $(a_1, a_2, \dots, a_k)^{\top} \in \Delta_k$, hence $\vec{x} \in \conv({\vec{x}}^{V_1}, {\vec{x}}^{V_2}, \dots, {\vec{x}}^{V_k})$.
\end{proof}
Choose now $\vec{x}\in \Delta_n$ and consider a partition $\mathcal{P}=\set{ V_1, V_2, \dots, V_k}$ of $\supp(\vec{x})$ separating distinct values of $\vec{x}$. Then $\vec{x} \in \linspan({\vec{x}}^{V_1}, {\vec{x}}^{V_2}, \dots, {\vec{x}}^{V_k})$, and by Proposition~\ref{prop:hull} it follows that $\vec{x}\in \conv({\vec{x}}^{V_1}, {\vec{x}}^{V_2}, \dots, {\vec{x}}^{V_k})$. Therefore, it makes sense to consider some barycentric coordinates of $\vec{x}$ associated with $\mathcal{P}$, that in general depend on how the elements of $\mathcal{P}$ are ordered.
\begin{definition}\label{def:enumeration}
    Given a family $\mathcal{F} \subseteq 2^{V}$, we call \emph{enumeration} of $\mathcal{F}$ any tuple $\tuplenum = (V_1, V_2, \dots, V_k)$ of \emph{distinct} subsets of $V$ such that $\mathcal{F} = \set{ V_1, V_2, \dots, V_k}$.
\end{definition}

 Continuing with the previous example, if $\tuplenum = ( V_1, V_2, \dots, V_k)$ is an enumeration  of $\mathcal{P}$, then $\bary_{\tuplenum}(\vec{x})$ is a well defined element of $\Delta_k$, and it is immediate to see that $\supp(\vec{x}) \subseteq \cup_{\ell = 1}^{k} V_{\ell}$, with equality if and only if $\bary_{\tuplenum}(\vec{x}) \in\relint{\Delta_k}$.

Moreover, if $\vec{x} \in \gkkt{c}$, then $\bary_{\tuplenum}(\vec{x})$ satisfies some additional algebraic relations that depend on the structure of $G$ and that involve the graph-theoretical notion of \emph{density} \cite{diestel_2005_graph_theory}. For every non-empty $S_1$, $S_2 \subseteq V$, let $e_G(S_1, S_2)$ count the \emph{ordered}
pairs of adjacent nodes in the set $S_1 \times S_2$, i.e., let
\begin{equation*}
        e_G(S_1,S_2) = \card{\set{ (i, j) \in S_{1}\times S_{2} : i \sim j}}.
\end{equation*}
By definition, $e_G(S_1,S_2)$ provides a way to count the edges crossing $S_1$ and $S_2$, but note that every edge with both ends in $S_1 \cap S_2$ is counted twice.
Call \emph{edge density} between $S_1$ and $S_2$ the ratio
    \begin{equation*}
        d_G(S_1,S_2) = \dfrac{e_G(S_1,S_2)}{\card{S_1} \card{S_2}}.
    \end{equation*}
In particular, given a family $\mathcal{F}$ of non-empty subsets of $V$, the edge densities between all possible pairs $(S_1 ,S_2) \in \mathcal{F} \times \mathcal{F}$ can be collected into a matrix, which in a sense sketches the connectivity between elements of $\mathcal{F}$.
Once again, an enumeration for $\mathcal{F}$ allows us to establish a unique way to build this matrix.
\begin{definition}\label{def:density-mat}
    For a tuple $\tuplenum=(V_1, V_2, \dots, V_k)$ of distinct non-empty subsets of $V$, the \emph{edge density matrix} associated with $\tuplenum$ is the \emph{symmetric} matrix $\mtx{D} \in \bbR^{k \times k}$ with general coefficient $d_{\ell, m} = d_G(V_\ell, V_m)$.
\end{definition}
Suitable edge densities and barycentric coordinates can be combined to produce an alternative expression for the multiplication by $\mtx{A}$.
\begin{lemma}\label{lem:multipliers-and-densities}
    Let $\tuplenum = ( V_1, V_2, \dots, V_k)$ be a tuple of non-empty and pairwise disjoint subsets of $V$, let $\vec{x}\in \conv({\vec{x}}^{V_1}, {\vec{x}}^{V_2}, \dots, {\vec{x}}^{V_k})$ and set $\vec{y} = \bary_{\tuplenum}(\vec{x})$. Then 
    \begin{equation*}
        (\mtx{A} \vec{x})_i
        = \sum_{m=1}^k d_G(\set{i}, V_m) y_m \qquad \text{for every $i \in V$}.
    \end{equation*}
\end{lemma}
\begin{proof}
    By hypothesis $\vec{x} = \sum_{\ell} y_\ell {\vec{x}}^{V_\ell}$, hence
    \begin{align*}
        \left( \mtx{A} \vec{x} \right)_i
        &= \sum_{j=1}^n a_{i j} x_j
        = \sum_{m=1}^k \sum_{j \in V_{m}} a_{i j} x_j
        = \sum_{m=1}^k \sum_{j \in V_{m}} a_{i j} y_{m}{\card{V_m}}^{-1}\\
        &= \sum_{m=1}^k \dfrac{e_G(\set{i}, V_m)}{\card{V_m}} y_m = \sum_{m=1}^k d_G(\set{i}, V_m) y_m.
    \end{align*} 
\end{proof}
Thanks to Lemma~\ref{lem:multipliers-and-densities}, we can prove that for $\vec{x} \in \gkkt{c}$, a certain vector $\vec{y}$ of barycentric coordinates for $\vec{x}$ is a KKT point for a corresponding quadratic program.
\begin{theorem}\label{thm:reduced-dynamics}
    Let $\vec{x} \in \Delta_{n}$ and let $\mathcal{P}$ be a partition of $\supp(\vec{x})$ separating distinct values of $\vec{x}$. Let $\tuplenum = (V_1, V_2, \dots, V_k)$ be an enumeration of $\mathcal{P}$, call $\mtx{D}$ the edge density matrix associated with $\tuplenum$ and set $\mtx{\Lambda} = \diag(\card{V_1}, \card{V_2}, \dots, \card{V_k})$.
    If $\vec{x} \in \gkkt{c}$,
    then $\bary_{\tuplenum}(\vec{x})$ is a KKT point for the program
    \begin{maxi}|l|[0]
        { \vec{y} \in \relint{\Delta_k}}
        {\QuadForm{(\mtx{D} + c \mtx{\Lambda}^{-1})}{\vec{y}}.}
        { \label{prog:reduced-c}}
        { }
    \end{maxi}
\end{theorem}
\begin{proof}
    Let $\vec{y} = \bary_{\tuplenum}(\vec{x})$.
    We may write $\vec{x} = \sum_{\ell} y_\ell {\vec{x}}^{V_\ell}$ by definition of $\vec{y}$.
    By Proposition~\ref{prop:kkt} and Lemma~\ref{lem:multipliers-and-densities}, there exists $\lambda$ such that for every $i \in \supp(\vec{x})$
    \begin{equation*}
    \lambda = \left(\left(\mtx{A} + c \mtx{I}\right) \vec{x}\right)_i = \Big(\sum_{m=1}^k d_G(\set{i}, V_m) y_m \Big) + c x_i.
    \end{equation*}
    For every $\ell \in \intupto{k}$, the arithmetic mean of the previous expression as $i$ varies in $V_{\ell}$ gives
    \begin{align*}
        \lambda &= \dfrac{1}{\card{V_\ell}}\sum_{i \in V_{\ell}}  \Big[ \Big( \sum_{m=1}^k d_G(\set{i}, V_m) y_m \Big) + c x_i \Big]\\
        &= \dfrac{1}{\card{V_\ell}}
        \sum_{i \in V_{\ell}} \bigg[ \bigg( \sum_{m=1}^k e_G(\set{i}, V_m) \dfrac{y_m}{\card{V_m}}\bigg) + c \dfrac{y_{\ell}}{ \card{V_{\ell}}} \bigg]\\  
        &= \Big( \sum_{m=1}^k d_G(V_{\ell}, V_m) y_m \Big) + c  \card{V_{\ell}}^{-1} y_{\ell} \\ 
        &= \left( \left( \mtx{D} + c \mtx{\Lambda}^{-1} \right) \vec{y} \right)_{\ell}.
    \end{align*}
    Then $\vec{y}$ is a KKT point for Program~\eqref{prog:reduced-c}.
\end{proof}
We remark that Eq.~\eqref{eq:bigcup-equiv} in the proof of Proposition~\ref{prop:gkkt-prog} is a trivial consequence of Theorem~\ref{thm:reduced-dynamics} when the partition considered is $\mathcal{P} = \set{\set{i} : i \in S}$.
\section{Equitable Partitions}\label{sec:hrf}
The converse of Theorem~\ref{thm:reduced-dynamics} would be helpful in characterizing the elements of $\gkkt{c}$. Unfortunately, the implication appearing in Theorem~\ref{thm:reduced-dynamics} cannot be replaced with a double implication. In fact, suppose $G$ is the graph on $V =\intupto{4}$ depicted in Fig.~\ref{fig:counterexample},
\begin{figure}[t] 
\centering
\includegraphics[width=0.18\textwidth]{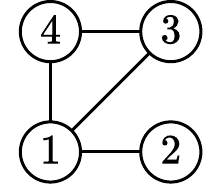} 
\caption{An instance for $G$ such that $\vec{x}^{\set{1,2,3,4}}$ is not an element of $\gkkt{c}$}\label{fig:counterexample}
\end{figure}
for which the adjacency matrix is
\begin{equation*}
    \mtx{A} = 
    \begin{pmatrix}
        0 & 1 & 1 & 1 \\
        1 & 0 & 0 & 0 \\
        1 & 0 & 0 & 1 \\
        1 & 0 & 1 & 0 
    \end{pmatrix}\!,
\end{equation*}
and consider the vector $\vec{x} = \big(\frac{1}{4}, \frac{1}{4}, \frac{1}{4}, \frac{1}{4}\big)^{\top}\!$, which by Proposition~\ref{prop:regular-bomze} is not an element of $\gkkt{c}$. For $V_1 = \set{1, 2}$ and $V_2 = \set{3,4}$, for which $\diag(V_1,V_2)^{-1} = \frac{1}{2}\mtx{I}$, the family $\mathcal{P}= \set{ V_1, V_2}$ with enumeration $\tuplenum = (V_1, V_2)$ has edge density matrix
\begin{equation*}
    \mtx{D} = 
    \dfrac{1}{2}
    \begin{pmatrix}
        1 & 1\\
        1 & 1
    \end{pmatrix},
\end{equation*}
and $\bary_{\tuplenum}(\vec{x}) = \big(\frac{1}{2}, \frac{1}{2}\big)^{\top}$ is a KKT point for the program
    \begin{maxi}|l|[0]
        { \vec{y} \in \relint{\Delta_2}}
        {\QuadForm{\left(\mtx{D} +  \dfrac{c}{2}\,\mtx{I}\right)}{\vec{y}}.}
        { \notag}
        { }
    \end{maxi}
However, Theorem~\ref{thm:reduced-dynamics} admits a partial converse under stronger hypotheses, which can be expressed using the notion of equitable partition \cite{godsil_agt_2001}.

\begin{definition}\label{def:highly}
    Given a graph $G'$ on a non-empty set of nodes $S$, a partition $\mathcal{P} = \set{V_1, V_2, \dots, V_k}$ of $S$ is called \emph{equitable} for $G'$ if for every $\ell$, $m \in \intupto{k}$ the number of neighbours in $V_{m}$ of a vertex $i$ in  $V_{\ell}$ is a constant $b_{\ell m}$, independent of $i$.
\end{definition}
Note that, if $\set{V_1, V_2, \dots, V_k}$ is an equitable partition for a graph $G'$, then in particular $V_m \neq \emptyset$ and $G'[V_{m}]$ is regular for every $\ell \in \intupto{k}$. The following proposition provides an equivalent formulation of Definition~\ref{def:highly} involving edge densities:

\begin{proposition}\label{prop:criterion}
    Given a graph $G'$ and a partition $\mathcal{P} = \set{V_1, V_2, \dots, V_k}$ of the set of nodes of $G'$, then $\mathcal{P}$ is equitable for $G'$ if and only if
    \begin{equation}\label{eq:highly-regular}
        d_{G'}(V_{\ell}, V_{m}) = d_{G'}(\set{i}, V_{m}) \qquad \text{ for every $\ell$, $m \in \intupto{k}$ and every $i \in V_{\ell}$.}
    \end{equation}
\end{proposition}
\begin{proof}
    Suppose first that $\mathcal{P}$ is an equitable partition. Fix $\ell$, $m \in \intupto{k}$ and some $i \in V_{\ell}$. Then
    \begin{equation}\label{eq:in-criterion}
        \sum_{j \in V_{\ell}} e_{G'}(\set{j}, V_m) = e_{G'}(V_{\ell}, V_{m}),
    \end{equation}
    and every term appearing in the summation is equal to $e_{G'}(\set{i}, V_{m})=b_{\ell m}$ by definition of equitable partition. Consequently, dividing both sides of Eq.~\eqref{eq:in-criterion} by $\card{V_{\ell}}\card{V_m}$ we get $d_{G'}(\set{i}, V_{m}) = d_{G'}(V_{\ell}, V_{m})$, and this value is independent of the choice of $i$ in $V_{\ell}$.
    Conversely, suppose Eq.~\eqref{eq:highly-regular} holds. Then every node in $V_{\ell}$ has $d_{G'}(V_{\ell}, V_m)\card{V_{m}}$ neighbors in $V_{m}$ for every $\ell$, $m \in \intupto{k}$, and this shows that $\mathcal{P}$ is equitable.
\end{proof}
Equitable partitions for induced subgraphs will play a key role in the remaining part of this paper.
\begin{proposition}\label{prop:highly}
    Consider some non-empty $S \subseteq V$. Then:
    \begin{mylist}
        \item\label{item:hrfe-trivial} The partition $\set{ \set{i} : i \in S}$ is equitable for $G[S]$;
        \item\label{item:hrfe-regular} The partition $\set{S}$ is equitable for $G[S]$ if and only if $G[S]$ is a regular graph;
        \item\label{item:hrfe-clique} If $S$ is \emph{not} an independent set, then \emph{every} partition of $S$ is equitable for $G[S]$ if and only if $S$ is a clique;
        \item\label{item:hrfe-auto} If $\mathcal{G}$ is a group of automorphisms for $G[S]$, then the orbits of $\mathcal{G}$ form an equitable partition for $G[S]$.
    \end{mylist}
\end{proposition}
\begin{proof}
    \ref{item:hrfe-trivial}, \ref{item:hrfe-regular}: Trivial by Definition~\ref{def:highly}. 
    
    \ref{item:hrfe-clique}: Note first that every partition of a complete graph is equitable. To prove the other implication, assume there exist two adjacent nodes $i_1$, $i_2 \in S$ and that every partition of $S$ is equitable for $G[S]$. 
    In particular, this is true for the partition $\set{ \set{ i_1 }, S \setminus \set{ i_1 } }$, and since $i_2 \sim i_1$, then $i \sim i_1$ for every $i \in S \setminus \set{ i_1 }$.
    Also $\set{ S }$ is an equitable partition for $G[S]$, thus $G[S]$ is a regular graph by \ref{item:hrfe-regular}, and so every $i \in S$ has $\card{S} - 1$ neighbors in $S$, i.e., $S$ is a clique.
    
    \ref{item:hrfe-auto}: This follows by \cite[p.~216, Exercise~2]{godsil_agt_2001}.
\end{proof}
As mentioned before, stronger assumptions enable us to strengthen the conclusions drawn in Theorem~\ref{thm:reduced-dynamics}.
\begin{theorem}\label{thm:reduced-dynamics-converse}
    Let $\vec{x} \in \Delta_{n}$ and let $\mathcal{P}$ be a partition of $\supp(\vec{x})$ separating distinct values of $\vec{x}$. Let $\tuplenum = (V_1, V_2, \dots, V_k)$ be an enumeration of $\mathcal{P}$, call $\mtx{D}$ the edge density matrix associated with $\tuplenum$ and set $\mtx{\Lambda} = \diag(\card{V_1}, \card{V_2}, \dots, \card{V_k} )$.
    Assume $\mathcal{P}$ is equitable for $G[\supp(\vec{x})]$. Then $\vec{x} \in \gkkt{c}$ if and only if 
    $\bary_{\tuplenum}(\vec{x})$ is a KKT point for Program~\eqref{prog:reduced-c}.
\end{theorem}
\begin{proof}
    One implication follows immediately by Theorem~\ref{thm:reduced-dynamics}.
    For the other implication, set $\vec{y} = \bary_{\tuplenum}(\vec{x})$ and
    assume  that $\vec{y}$ is a KKT point for Program~\eqref{prog:reduced-c}. Then $\supp(\vec{y}) = \intupto{k}$, and there exists $\lambda\in \bbR$ such that for every $\ell \in \intupto{k}$:
    \begin{equation*}
        ( ( \mtx{D}+ c \mtx{\Lambda}^{-1}) \vec{y} )_{\ell} = \lambda.
    \end{equation*}
    Pick any $i \in \supp(\vec{x})$. The node $i$ is in $V_{\ell}$ for some $\ell \in \intupto{k}$, and $d_G(\set{i}, V_m)= d_G(V_{\ell}, V_m)$ by Proposition~\ref{prop:criterion}. By Lemma~\ref{lem:multipliers-and-densities},
    \begin{equation*}        
        ((\mtx{A} + c \mtx{I} )\vec{x})_i = \Big(\sum_{m=1}^k d_G(\set{i}, V_m) y_m\Big) + c x_i = \Big(\sum_{m=1}^k d_G(V_{\ell}, V_m) y_m\Big) + c\card{V_\ell}^{-1}y_\ell = \lambda.
    \end{equation*}
    Then $\vec{x} \in \gkkt{c}$ by Proposition~\ref{prop:kkt}. 
\end{proof}
\section{Application to Equitable Bipartitions}\label{sec:binary}
Let $\mathcal{P} = \set{V_1, V_2}$ be a partition for some subset of $V$, define $\tuplenum = ( V_1, V_2 )$ and assume that $\mathcal{P}$ is an equitable partition for $G[V_1 \cup V_2]$.
Moreover, let $\mtx{D} = [d_{ij}]_{i,j}$ be the edge density matrix associated with $\tuplenum$ and set
\begin{equation*}
    \begin{cases}
        \alpha_1(\tuplenum) = \card{V_2}(d_{12}-d_{22}),\\
        \alpha_2(\tuplenum) =\card{V_1} (d_{21} - d_{11}).
        \end{cases}
\end{equation*}
\begin{proposition}\label{prop:technical}
    $G[V_1 \cup V_2]$ is a regular graph if and only if $\alpha_1(\tuplenum) = \alpha_2(\tuplenum)$.
\end{proposition}
\begin{proof}
    For $\ell = 1$, $2$, every node in $V_\ell$ has $d_{\ell 1}\card{V_1} + d_{\ell 2}\card{V_2}$ neighbors in $V_1 \cup V_2$. Then $G[V_1 \cup V_2]$ is regular if and only if $d_{11}\card{V_1} + d_{12}\card{V_2} = d_{21}\card{V_1} + d_{22}\card{V_2}$, which is equivalent to $\alpha_1(\tuplenum) = \alpha_2(\tuplenum)$.
\end{proof}

The regularity of $G[V_1 \cup V_2]$ plays a key role in how $\gkkt{c}$ intersects $[\vec{x}^{V_1}, \vec{x}^{V_2}]$.
\begin{corollary}\label{cor:case-v1-v2-regular}
    Assume $G[V_1 \cup V_2]$ is a regular graph.
    Then $c^{*} = \alpha_1(\tuplenum) = \alpha_2 (\tuplenum)$ is such that
    \begin{itemize}
        \item If $c = c^{*}$,
        then $\gkkt{c} \cap [\vec{x}^{V_1}, \vec{x}^{V_2}]  = [\vec{x}^{V_1}, \vec{x}^{V_2}]$;
        \item If $c \neq c^{*}$,
            then $\gkkt{c} \cap [\vec{x}^{V_1}, \vec{x}^{V_2}] =
            \set{ \vec{x}^{V_1}, \vec{x}^{V_2}, \vec{x}^{V_1 \cup V_2}}$.
    \end{itemize}
\end{corollary}
\begin{proof}
    By Proposition~\ref{prop:technical}, $\alpha_1(\tuplenum) = \alpha_2 (\tuplenum)$, hence $c^{*}$ is well defined.
    Both $\vec{x}^{V_1}$ and $\vec{x}^{V_2}$ are in $\gkkt{c}$ by Proposition~\ref{prop:regular-bomze}, and 
    by Theorem~\ref{thm:reduced-dynamics-converse} we can find the remaining elements of $\gkkt{c}$ within $[\vec{x}^{V_1}, \vec{x}^{V_2}]$ by looking for points of the form $y_1 \vec{x}^{V_1} + y_2\vec{x}^{V_2}$, where $(y_1, y_2)^{\top} \in \relint{\Delta_2}$ satisfies for some $\lambda \in \bbR$
    \begin{equation*}
        \left( \mtx{D} + c\diag\left(\card{V_1}^{-1}, \card{V_2}^{-1}\right) \right)
        \begin{pmatrix} y_1 \\ y_2 \end{pmatrix} = \begin{pmatrix}
        \lambda \\ \lambda
        \end{pmatrix}.
    \end{equation*}
    By eliminating $\lambda$, this means that
    \begin{equation*}
        (c\card{V_1}^{-1} + d_{11}) y_1 + d_{12} \, y_2 
        = d_{21}\, y_1 + (d_{22} + c\card{V_2}^{-1}) y_2,
    \end{equation*}
    which is equivalent to
    \begin{equation}\label{eq:alpha-beta}
        (c - \alpha_2(\tuplenum) )y_1 \card{V_1}^{-1}
        = (c - \alpha_1(\tuplenum) )y_2 \card{V_2}^{-1}.
    \end{equation}
    Then for $c = c^*$ every $(y_1, y_2)^{\top} \in \relint{\Delta_2}$ satisfies Eq.~\eqref{eq:alpha-beta}. For $c \neq c^*$, dividing both sides in Eq.~\eqref{eq:alpha-beta} by $c - c^{*}$ yields $y_1 \card{V_1}^{-1}
        = y_2 \card{V_2}^{-1}$, leading to the solution $\vec{x}^{V_1 \cup V_2}$.  
\end{proof}

\begin{corollary}\label{cor:case-v1-v2-not-regular}
    Assume $G[V_1 \cup V_2]$ is \emph{not} a regular graph.
    Then the closed interval $I = \conv(\alpha_1(\tuplenum), \alpha_2(\tuplenum))$ is such that:
    \begin{itemize}
        \item If $c \in I$, then $\gkkt{c} \cap [\vec{x}^{V_1}, \vec{x}^{V_2}] 
        = \set{ \vec{x}^{V_1}, \vec{x}^{V_2} }$;
        \item If $c  \not\in I$, then $\gkkt{c} \cap [\vec{x}^{V_1}, \vec{x}^{V_2}]  
        = \set{ \vec{x}^{V_1}, \vec{x}^{V_2}, \vec{x}(c)}$, where
        \begin{equation}\label{eq:x-di-c}
            \vec{x}(c) = \sum_{\ell = 1}^2 \dfrac{(c - \alpha_{\ell}(\tuplenum))|V_{\ell}|}{(c - \alpha_1(\tuplenum))|V_1| + (c - \alpha_2(\tuplenum))|V_2|} \vec{x}^{V_{\ell}}
        \end{equation}
        satisfies $\supp{\vec{x}(c)} = V_1 \cup V_2$ and $\vec{x}(c) \neq \vec{x}^{V_1 \cup V_2}$.
    \end{itemize}    
\end{corollary}
\begin{proof}
    As before, both $\vec{x}^{V_1}$ and $\vec{x}^{V_2}$ are in $\gkkt{c}$ regardless of the value of $c$ by Proposition~\ref{prop:regular-bomze}.
    However, this time $\alpha_1(\tuplenum) \neq \alpha_2 (\tuplenum)$ by Proposition~\ref{prop:technical}, and so $I$ is a closed interval with distinct endpoints.
    Let $\vec{x} \in [\vec{x}^{V_1}, \vec{x}^{V_2}]$ such that $\vec{x} \neq \vec{x}^{V_1}$, $\vec{x}^{V_2}$.
    Arguing as in the proof of Corollary~\ref{cor:case-v1-v2-regular}, then $\vec{x} \in \gkkt{c} \cap [\vec{x}^{V_1}, \vec{x}^{V_2}]$ if and only if $\vec{x} = y_1 \vec{x}^{V_1} + y_2\vec{x}^{V_2}$, where $(y_1, y_2)^{\top} \in \relint{\Delta_2}$ satisfies Eq.~\eqref{eq:alpha-beta}.
    For positive $y_1$, $y_2$ a solution to Eq.~\eqref{eq:alpha-beta} exists only in case $c \not \in I$, and if that occurs, then it is easy to check that $\vec{x} = \vec{x}(c)$.
\end{proof}
Recall that a \emph{star graph} is a complete bipartite graph in which one node, called center, is an end of every edge in the graph \cite{diestel_2005_graph_theory}.
\begin{definition}\label{def:gen_star}
    We say that a graph $G'=(V',E')$ is a \emph{generalized star} with \emph{core}%
    \footnote{The term \emph{core} has a different meaning in algebraic graph theory, see e.g. \cite[p.~104]{godsil_agt_2001}.}
    $H'$ if: 
    \begin{enumerate}[label = \textup{(GS\arabic*)},left = 0pt]
        \item \label{i:gs1} $H'$ is a proper non-empty subset of $V'$;
        \item \label{i:gs2} Every node in $H'$ is adjacent to every node in $V' \setminus H'$;
        \item \label{i:gs3} $H'$ is a clique, but $V'$ is not a clique;
        \item \label{i:gs4} The induced subgraph $G[V' \setminus H']$ is regular, but not a clique.
    \end{enumerate}
\end{definition}
\begin{proposition}
    Let $G'=(V',E')$ be a star graph with center $u \in V'$. If $\card{V'} \geq 3$, then $G'$ is a generalized star with core $\set{u}$.
\end{proposition}
\begin{proof}
    A star graph that has three or more nodes is not a clique. It is then immediate to see that Definition~\ref{def:gen_star} is satisfied.
\end{proof}
Note that the cherry graph (Fig.~\ref{fig:cherry}) is a star and admits the equitable partition $\set{\set{1,2}, \set{3}}$. In this case, $\vec{x}(c)$ given by Corollary~\ref{cor:case-v1-v2-not-regular} for $c = 0$ is precisely the spurious solution $\big(\frac{1}{4}, \frac{1}{4}, \frac{1}{2}\big)^{\top}$. Similarly, Corollary~\ref{cor:case-v1-v2-not-regular} can be applied to every star graph with at least $3$ nodes, as well as to generalized star graphs.
\begin{theorem}\label{thm:genstar}
    Let $H$, $P$ be disjoint subsets of $V$ such that $G[H \cup P]$ is a generalized star with core $H$.
    Set $h = \card{H}$, $p = \card{P}$ and assume $G[P]$ is a $d$-regular graph.
    Then $b =  p - d > 1$, and for $c \not\in [1, b]$ and 
    \begin{equation}\label{eq:gen-star-y}
    \begin{cases}        
        y_1 &= (c - 1) p \,[(c - 1)p + (c - b) h]^{-1},\\
        y_2 &= (c - b) h \,[(c - 1)p + (c - b) h]^{-1},
    \end{cases}
    \end{equation}    
    the vector $\vec{x} = y_1\vec{x}^{P} + y_2\vec{x}^{H}$ is an element of $\gkkt{c}$ with support $H \cup P$.
\end{theorem}
\begin{proof}
    By hypothesis, $\mathcal{F} = \set{H, P }$ is an equitable partition for $G[H \cup P]$, and the edge density matrix associated with $(P, H)$ is
    \begin{equation*}
        \mtx{D} =
        \begin{pmatrix}
        d/p\quad & 1\\
        1 \quad & 1 - 1/h
        \end{pmatrix}.
    \end{equation*}
    A direct computation shows that
    \begin{equation*}
        \begin{cases}
            \alpha_1((P, H)) = h[1 - (1 - h^{-1})] = 1,\\
            \alpha_2((P, H)) = p( 1 - d p^{-1}) = p - d = b.
            \end{cases}
    \end{equation*}
    Moreover, $G[P]$ is not complete by \ref{i:gs3}, thus $d < p - 1$ and so $1 < b \leq p$.
    Finally, if $c \in \bbR \setminus [1, b]$, then the thesis follows by Corollary~\ref{cor:case-v1-v2-not-regular} applied to $(V_1,V_2) = (P, H)$.
 \end{proof}

In \cite[Theorem~10]{pelillo_jagota_1996}, a configuration of cliques ${C_1}$, ${C_2}$, $\dots$, ${C_q}$ of equal cardinality is exhibited such that $\conv(\vec{x}^{C_1}, \dots, \vec{x}^{C_q})$ is entirely contained in $\gkkt{0}$. To show this, the authors prove that every point of that specific convex hull is a local solution for the Motzkin-Straus QP. Recently, this result has been generalized in \cite{tang_maxima_2022}.

As an application of Theorem~\ref{thm:genstar}, we are able to exhibit a particular configuration of cliques ${C_1}$, ${C_2}$, $\dots$, ${C_q}$ and conditions on $c$ such that at least one vector with support $\cup_{\ell} C_{\ell}$ is an element of $\gkkt{c} \setminus \conv(\vec{x}^{C_1}, \dots, \vec{x}^{C_q})$.

\begin{figure}[t] 
\centering
\includegraphics[width=0.7\textwidth]{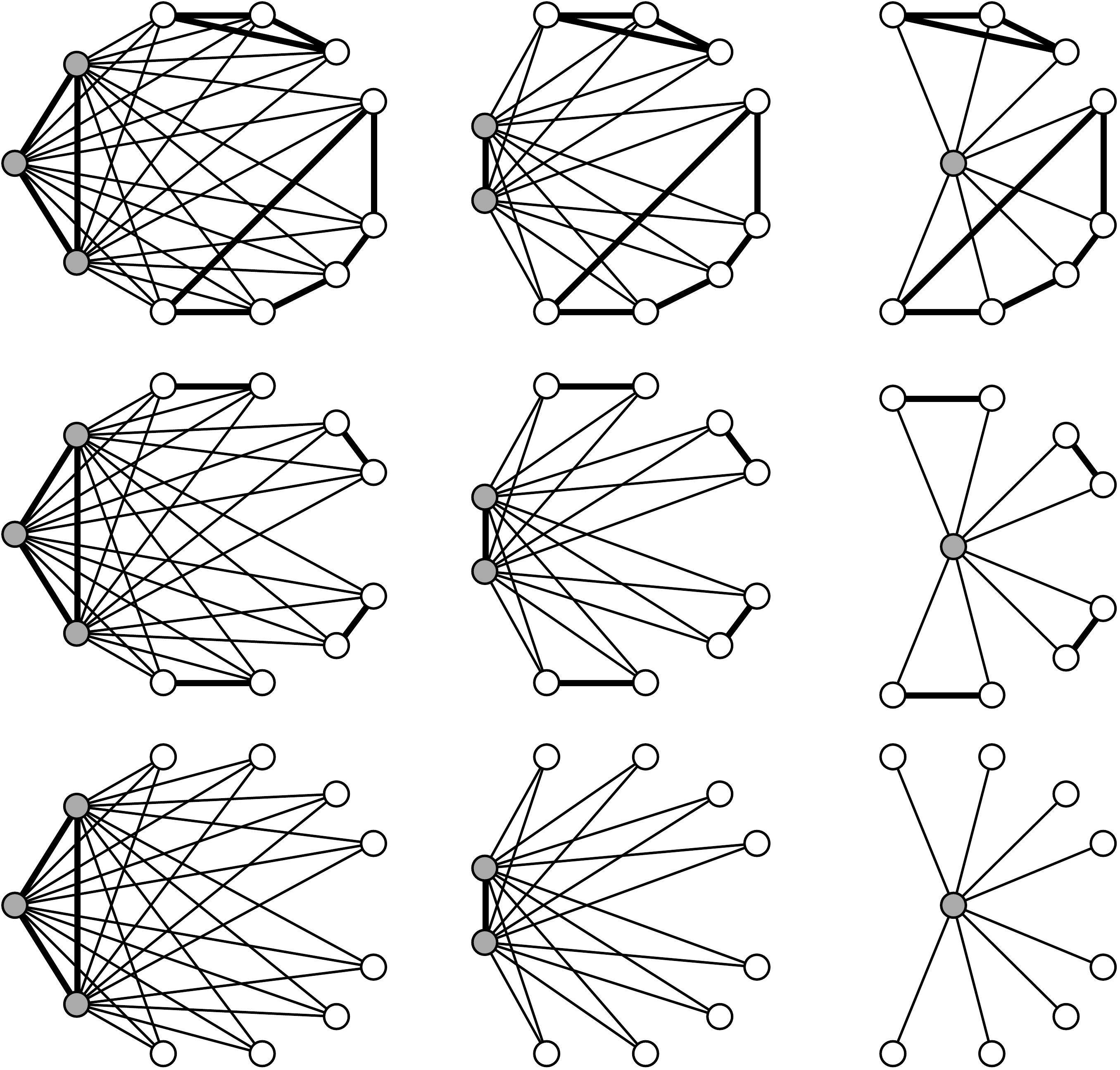} 
\caption{Some generalized stars, with core nodes represented in gray}\label{fig:flower}
\end{figure}

\begin{corollary}\label{cor:one-shared-set}
    Consider $q \geq 2$ distinct cliques $C_1$, $C_2$, $\dots$, $C_q$ such that:
    \begin{enumerate}[label = \textup{(\roman*)},left = 0pt]
        \item \label{i:shared-size} There exists an integer $s \geq 2$ such that $\card{C_{\ell}}= s$ for every $\ell \in \intupto{q}$;
        \item \label{i:shared-sunflower} The set $H = \cap_{\ell} C_\ell$ is not empty and $C_{\ell} \cap C_{m} = H$ for every distinct $\ell$, $m \in \intupto{q}$;
        \item \label{i:shared-hr} The set $P =\cup_\ell C_{\ell}\setminus H$ is not empty and $G[P]$ is regular but not complete.
    \end{enumerate}
    There exist $c_0 < 1$ and $1 < b \leq \card{P}$ such that,
    if $c \not\in \set{c_0} \cup [1, b] $, 
    then at least one vector with support $\cup_{\ell} C_{\ell}$ lies in $\gkkt{c} \setminus \conv(\vec{x}^{C_1}, \dots, \vec{x}^{C_q})$.
\end{corollary}
\begin{proof}
    Observe that $G[H \cup P]$ is a generalized star with core $H$. 
    Define $h = \card{H}$, $p = \card{P}$, and assume $G[P]$ is a $d$-regular graph.
    For $b = p - d$ and $c \in \bbR \setminus [1, b]$, define $y_1$ and $y_2$ as in Eq.~\eqref{eq:gen-star-y} so that, by Theorem~\ref{thm:genstar}, the vector $\vec{x} = y_1\vec{x}^{P} + y_2\vec{x}^{H}$ is an element of $\gkkt{c}$.
    It will be useful to note that dividing termwise the equalities in Eq.~\eqref{eq:gen-star-y} yields 
    \begin{equation}\label{eq:ratio}
        \dfrac{y_2}{h} : \dfrac{y_1}{p} = \dfrac{c-b}{c-1}.
    \end{equation}
    Now, suppose it is possible to write $\vec{x}$ as a convex combination of $\vec{x}^{C_1}$, $\vec{x}^{C_2}$, $\dots$, $\vec{x}^{C_q}$.
    In this case, there is an alternative way to compute the left-hand side of Eq.~\eqref{eq:ratio}. In fact, note that $x_i =x_{j}$ whenever $i$, $j \in P$ by construction, and due to hypothesis \ref{i:shared-size} there exists only one element of $\conv(\vec{x}^{C_1}, \dots, \vec{x}^{C_q})$ with this property, namely the arithmetic mean of $\vec{x}^{C_1}$, $\vec{x}^{C_2}$, $\dots$, $\vec{x}^{C_q}$. Therefore
    \begin{equation*}
    \vec{x} = \dfrac{1}{q} \sum_{\ell = 1}^q\vec{x}^{C_{\ell}}, 
    \end{equation*}
    hence
    \begin{equation*}
        x_i = 
        \begin{cases}
            q^{-1}s^{-1}    & \quad\text{$i \in P$,}\\
            s^{-1}          & \quad\text{$i \in H$,}\\
            0               & \quad\text{otherwise.}
        \end{cases}
    \end{equation*}
    Consequently, $y_1 = p q^{-1}s^{-1}$ and $y_2 = h s^{-1}$ and so
    \begin{equation}\label{eq:ratio-II}
        \dfrac{y_2}{h} : \dfrac{y_1}{p} = q.
    \end{equation}
    By Eq.~\eqref{eq:ratio} and Eq.~\eqref{eq:ratio-II} we get $(c - b) ( c - 1 )^{-1} = q$, i.e., $c = (q - b) (q - 1)^{-1}$.
    To complete the proof, define $c_0 = (q - b) (q - 1)^{-1}$ and note that $b \geq 2$ implies $c_0 < 1$.  
\end{proof}
\section{Replicator Dynamics}\label{sec:rd}
Consider a matrix $\mtx{M} \in \bbR^{n \times n}$ and the associated ordinary differential equation
\begin{equation}\label{eq:rd_cont}
    \dot{x}_i = x_i\left[(\mtx{M}\vec{x})_i - \vec{x}^\top\mtx{M}\vec{x}\right], \qquad \text{for every }i \in \intupto{n}.
\end{equation}
It is known that $\Delta_n$ is invariant under the flow defined by Eq.~\eqref{eq:rd_cont}, see e.g. \cite{Hofbauer_Sigmund_Egap_1998}, and the corresponding dynamical system defined on $\Delta_n$ by Eq.~\eqref{eq:rd_cont} is known as the continuous-time \emph{replicator dynamics}
with payoff-matrix $\mtx{M}$. We recall that a point $\vec{z}\in \Delta_n$ is \emph{stationary} for Eq.~\eqref{eq:rd_cont} if the constant function $\vec{x}(t) = \vec{z}$ is a trajectory under Eq.~\eqref{eq:rd_cont}.
Bomze provided in \cite{bomze_1997} a characterization of stationary points for Eq.~\eqref{eq:rd_cont}.
In particular, this characterization can be rephrased as follows when specialized to the case $\mtx{M} = \mtx{A} + c \mtx{I}$:
\begin{theorem}\label{thm:replicator}
    A point $\vec{x} \in \Delta_n$ is stationary for the replicator dynamics with payoff-matrix $\mtx{A} + c \mtx{I}$ if and only if $\vec{x} \in \gkkt{c}$.
\end{theorem}
\begin{proof}
    Imposing $\dot{\vec{x}}=0$ in Eq.~\eqref{eq:rd_cont}, it is easy to see that a point $\vec{x} \in \Delta_n$ is stationary for the replicator dynamics with payoff-matrix $\mtx{M}$ if and only if there exists some $\lambda \in \bbR$ such that $(\mtx{M}\vec{x})_i = \lambda$ for every $i \in \supp(\vec{x})$. By Proposition~\ref{prop:kkt}, this is equivalent to $\vec{x} \in \gkkt{c}$ in case $\mtx{M} = \mtx{A} + c \mtx{I}$.
\end{proof}
Thanks to Theorem~\ref{thm:replicator}, every result presented in this paper about $\gkkt{c}$ can be restated as a result about stationary points for the replicator dynamics with payoff-matrix $\mtx{A} + c \mtx{I}$.
This can be used, for instance, to describe the following consequence of Proposition~\ref{prop:one-at-most}. Let $\vec{x} \in \Delta_n$ and assume $\vec{x}$ is not a characteristic vector. As reported in \cite{bomze_annealed_2002,pelillo_torsello_payoff_monotonic_2006}, an analysis on the spectrum of $\mtx{A} + c \mtx{I}$ shows that if $c$ lies in a certain range, which depends on both $\supp(\vec{x})$ and $\mtx{A}$, then $\vec{x}$ cannot be a stationary point for the replicator dynamics with payoff-matrix $\mtx{A} + c \mtx{I}$. This information is applied in \cite{bomze_annealed_2002,pelillo_torsello_payoff_monotonic_2006} to derive annealing procedures on $c$ providing heuristics for the maximum clique problem --- see also \cite{daniluk_ImplementationMaximumClique_2019} for a parallelized implementation.
By Proposition~\ref{prop:one-at-most}, we can say more: since $\vec{x}$ lies in $\gkkt{\gamma}$ for at most one value of $\gamma \in \bbR$, if $\vec{x}$ is a stationary point for the replicator dynamics with payoff-matrix $\mtx{A} + c \mtx{I}$, then $\vec{x}$ is no longer a stationary point if we perturb $c$ to change its value, and this is true for \emph{any} non-null perturbation on $c$.

Moreover, the following theorem reports additional known results, which can be proved using that the matrix $\mtx{A} + c \mtx{I}$ is \emph{symmetric}, see e.g. \cite[Theorem~2]{bomzeRegularityDegeneracyDynamics2002}, \cite[Theorem~3, 
 Lemma~4]{bomze_1997} and \cite{akin_RecurrenceUnfit_1982}:
\begin{theorem}\label{thm:solver}
    Let $\vec{x}(t)$ be a non-stationary trajectory under Eq.~\eqref{eq:rd_cont} for $\mtx{M} = \mtx{A} + c \mtx{I}$ and such that $\vec{x}(0) \in \Delta_n$. Then:
    \begin{mylist}
        \item The function $(f_c \circ \vec{x})(t) = \vec{x}^{\top}(t)(\mtx{A} + c \mtx{I})\vec{x}(t)$ is strictly increasing in $t$;
        \item There exists $\vec{z} = \lim_{t \to +\infty}\vec{x}(t)$ and  $\vec{z} \in \gkkt{c}$;
        \item If $\vec{x}(0) \in \relint{\Delta_n}$ then $\vec{z} \in \kkt{c}$.
    \end{mylist}
\end{theorem}
Theorem~\ref{thm:solver} motivates the use of replicator dynamics to solve the generalized Motzkin-Straus QP \cite{bomze_1997,bomze_annealed_2002,pelillo_torsello_payoff_monotonic_2006,pelillo_ReplDynaComb_2008,daniluk_ImplementationMaximumClique_2019} since it guarantees that the replicator dynamics trajectories initialized in $\relint{\Delta_n}$ will always converge to a KKT point of the program.
However, although such trajectories will never hit a (proper) face in finite time (see e.g. \cite[p.~119]{weibull_1995_EGT}), in practical numerical implementations they might be subject to additional complications. In fact, a ``jump'' into a face of $\Delta_n$ may occur for various reasons, such as approximation errors due to floating-point arithmetic or thresholding procedures, and sometimes it might even be an explicit design choice to speed up convergence.\footnote{Bomze \cite{bomze_LotkaVolterraEquationReplicator_1983} shows several examples of interior trajectories that, after getting very close to a simplex face, suddenly move away from it.  In such cases, approximation or thresholding issues as the ones described above might result in an improper convergence of the algorithm.}
We remark that once a trajectory hits a face it cannot move away from it, and if that face does not contain a KKT point then the process is guaranteed to converge {\em only} to a generalized KKT point.
In this respect, our results can be seen as a contribution towards understanding the properties of the limit points of such ill-behaved (simulated) trajectories, in terms of properties of the subgraphs induced by their support.
\section{Conclusion}
In this article, we discussed properties of the generalized Karush-Kuhn-Tucker points associated with a parametric version of the Motzkin-Straus QP introduced by Bomze.
We extended some known results regarding characteristic vectors and generalized KKT points, and provided new insights into these points and the symmetries of the subgraphs induced by their support.
Thanks to a suitable use of barycentric coordinates, we studied the generalized KKT points by representing them as convex combinations of characteristic vectors. Finally, considering equitable bipartitions for induced subgraphs, we identified some generalized KKT points associated with generalized star graphs and related to the well-known spurious solution of the Motzkin-Straus QP for the cherry graph.

We conclude by suggesting potential research directions for future work. On the one hand, it would be interesting to find a weaker version of Theorem~\ref{thm:reduced-dynamics-converse} that does not involve equitable partitions, as it could lead to a more precise description of generalized KKT points.
On the other hand, an interesting direction is related to the replicator dynamics.
In particular, Theorem~\ref{thm:reduced-dynamics} and Theorem~\ref{thm:reduced-dynamics-converse} provide a correspondence between stationary points associated with two distinct replicator dynamics, evolving on simplices having (in general) different affine dimension. 
It would be interesting to further explore this correspondence and to understand whether it may simplify the search for replicator dynamics stationary points
via a suitable dimensionality reduction. Additionally, it would be interesting to investigate  to what extent the consequence of Proposition~\ref{prop:one-at-most} mentioned in  Sect.~\ref{sec:rd} may be useful to improve the annealing procedures described in \cite{bomze_annealed_2002,pelillo_torsello_payoff_monotonic_2006,daniluk_ImplementationMaximumClique_2019}.
Finally, as explained in Sect.~\ref{sec:rd}, numerical implementations of the replicator dynamics with payoff matrix $\mtx{A}+c \mtx{I}$ may return vectors of $\gkkt{c}$ that are not KKT points for the generalized Motzkin-Straus QP, despite the theory regarding the trajectories initialized in $\relint{\Delta_n}$. The results presented in this paper suggest that these vectors could still offer insights into the underlying graph structure. It remains an open question, however, whether this kind of information might be useful in practical applications.
\section*{Acknowledgments}
The authors thank the two anonymous reviewers for their helpful suggestions and Immanuel~M.~Bomze for providing valuable bibliographic recommendations. Moreover, the authors would like to thank Morteza H.~Chehreghani and Sebastiano Smaniotto for their help during the initial stages of this work.

\section*{Declarations}
\subsection*{Statements and Declarations}
\bmhead{Funding} Guglielmo Beretta's scholarship is funded jointly by Ca' Foscari University of Venice and by Polytechnic University of Turin.
\bmhead{Competing interests} The authors have no competing interests to declare that are relevant to the content of this article.
\bmhead{Ethics approval} Not applicable.
\bmhead{Consent to participate} Not applicable.
\bmhead{Consent for publication} Not applicable.
\bmhead{Availability of data and materials} Not applicable.
\bmhead{Code availability} Not applicable.
\bmhead{Authors' contributions} 
All authors contributed to the study conception and design. The first draft of the manuscript was written jointly by Alessandro Torcinovich and Guglielmo Beretta. This study has been performed under the supervision of Marcello Pelillo. All authors commented on previous versions of the manuscript. All authors read and approved the final manuscript.
\bmhead{Artwork} Figure~\ref{fig:cherry}, Figure~\ref{fig:counterexample} and Figure~\ref{fig:flower} have been produced using \texttt{draw.io}.
\bmhead{Conference Paper} A shorter version of this paper has been presented at the 7th International Conference on the Dynamics of Information Systems (DIS 2024) and will appear in the conference proceedings in the series Lecture Notes in Computer Science, Springer.

\backmatter

\end{document}